\documentclass[10pt,reqno]{amsart}
\usepackage{amssymb,graphicx,color}
\usepackage[notcite,notref]{}
\usepackage{multirow}
\usepackage{booktabs,threeparttable,caption2}

\usepackage{subfigure}\allowdisplaybreaks
\catcode`\@=11
\@addtoreset{equation}{section}   

\@addtoreset{table}{section}   
\renewcommand\thefigure{\thesection.\@arabic\c@figure}
\@addtoreset{figure}{section}   
\renewcommand\thetable{\thesection.\@arabic\c@table}

 \newcommand{\new}{\newcommand*}
 \new{\rnew}{\renewcommand*}
 \new{\newe}{\newenvironment*}
 \new{\stl}{\setlength}
 \stl{\arraycolsep}{0.5mm}

\newcommand{\refl}[1]{Lemma~{\rm \ref{#1}}}
\newcommand{\reft}[1]{Theorem~{\rm \ref{#1}}}
\newtheorem{thm}{\bf Theorem}

\newenvironment{theorem}{\begin{thm}} {\end{thm}}
\newtheorem{lmm}{\bf Lemma}

\newenvironment{lemma}{\begin{lmm}}{\end{lmm}}
\theoremstyle{remark}

\theoremstyle{definition}

\newcommand{\refe}[1]{{\rm (\ref{#1})}}
\newcommand {\bgeq}[1]{\begin{equation}\label{#1}}
\newcommand \edeq {\end{equation}}
\newcommand \bgth {\begin{theorem}\label}
\newcommand \edth {\end{theorem}}
\newcommand \bglm {\begin{lemma}\label}
\newcommand \edlm {\end{lemma}}
\newcommand {\bgar}[1]{\begin{array}{#1}}
\newcommand {\edar}{\end{array}}

\newcommand \Lbd {\Lambda}

\newcommand \af  {\alpha}
\newcommand \bt  {\beta}
\newcommand \gm   {\gamma}
\newcommand \dt  {\delta}

\newcommand \dint {\displaystyle \int}
\newcommand \dsum {\displaystyle\sum}

\def \nns {\\[6pt]}

\def \lbd {\lambda}

\def \Om {\Omega}

\renewcommand \Om {\Lbd}

\begin{document}
\bibliographystyle{plain}
\title[Laguerre spectral method]
{A fully diagonalized spectral method using generalized Laguerre functions on the half line
}
\author[F. Liu]{Fu-jun Liu$^{1}$}
\thanks{$^{1}$Department of Mathematics, Shanghai Normal
University, Shanghai, 200234, China; School of Science, Henan Institute of Engineering, Zhengzhou, 451191, China. Email: liufujun1981@126.com}

\author[Z. Wang]{Zhong-qing Wang$^{2}$}
\thanks{$^{2}$School of Science, University of Shanghai for Science and Technology, Shanghai, 200093, China. Email: zqwang@usst.edu.cn}

\author[H. Li]{Hui-yuan Li$^{3}$}
\thanks{$^{3}$State Key Laboratory of Computer Science/Laboratory of Parallel Computing,  Institute of Software, Chinese Academy of Sciences, Beijing 100190, China. Email: huiyuan@iscas.ac.cn}


\thanks{The first author was supported by Science and Technology Research Program of Education Department of Henan Province (No. 13A110005);
The second author was supported in part by National Natural Science
Foundation of China (No. 11571238) and the Research Fund for
Doctoral Program of Higher Education of China (No. 20133127110006);
The third author was supported by National Natural Science Foundation of China (Nos. 91130014, 11471312 and 91430216).}
\begin{abstract}
A fully diagonalized spectral method using generalized Laguerre functions
is proposed and analyzed  for solving elliptic equations on the half line.
We  first define  the generalized Laguerre functions which are complete and
mutually orthogonal  with respect to an equivalent Sobolev inner product.
Then the Fourier-like Sobolev orthogonal basis functions are constructed
for the  diagonalized Laguerre spectral method  of elliptic equations.
Besides,
a unified  orthogonal Laguerre projection is established
for various  elliptic equations.
On the basis of this orthogonal Laguerre projection, we obtain optimal error  estimates
of the fully diagonalized Laguerre spectral method for both Dirichlet and Robin
boundary value problems. Finally, numerical experiments, which are in agreement with the
theoretical analysis, demonstrate the effectiveness and the spectral accuracy of
our diagonalized method.
\end{abstract}
\maketitle

\section{Introduction}
Spectral methods for solving partial differential equations on unbounded domains
 have gained a rapid development
during  the last few decades.  An abundance  of literature on this research topic
has emerged, and their underlying  approximation approaches
can  be essentially classified into three catalogues  \cite{Boydbook,ShenWang}: 

\begin{itemize}
\item[(i)] truncate an unbounded domain to a bounded one and solve the problem on the bounded domain subject to  artificial or transparent boundary conditions \cite{NS,SW07};
\item[(ii)]  map the original problem on an unbounded domain  to  one on a bounded domain and use classic  spectral methods to solve the new problem \cite{G00}; or equivalently, approximate the  original problem  by  some  non-classical functions  mapped  from the classic orthogonal
polynomials/functions on a bounded domain \cite{Boyd87a,Boyd87b,Cris82,GSW00,GSW02,ShenWang,WG08,YG12};
\item[(iii)]  directly approximate the original problem by genuine   orthogonal functions such as  Laguerre polynomials or functions on the unbounded domain \cite{Coulaud,GS,GuoShenXu,GuoSunZhang,GWW,GuoZhang,GuoZhang1,IF,MaG,MPV,Shen,WGW,WX15,XG1,ZSX}.
\end{itemize}

The third approach is of particular interest to researchers, and has won an increasing popularity in a broad class of  applications, owing to its essential advantages over other two approaches.
These direct approximation schemes constitute an initial step
towards the efficient spectral methods,  which admit  fast and stable algorithms
for their efficient implementations.

As we know,
the Fourier spectral method makes use of
 the eigenfunctions of the Laplace operator which are orthogonal to each other with respect to the Sobolev inner product  involving derivatives, thus  the corresponding algebraic system is diagonal \cite{Boydbook,CHQZ06,STW}.
 This fact together with  the availability of the fast Fourier transform (FFT) makes
 the Fourier spectral method be an ideal approximation approach  for differential equations
 with periodic boundary conditions.
 Although the utilization of the genuine orthogonal polynomials/functions in this direct approach usually  leads to a highly sparse (e.g., tri-diagonal, penta-diagonal) and well-conditioned algebraic system, however, in many cases, people still want to get
a set of Fourier-like basis functions for a fully diagonalized algebraic system \cite{ShenWang07}.

The main purpose  of this paper is to construct the   Fourier-like Sobolev orthogonal basis functions
\cite{FerMarPinXu,MarXu}  for elliptic boundary value problems on the half line $\Lambda=(0,\infty)$.  For this purpose, we shall first extend the definition of
 Laguerre polynomials $\big\{ \mathcal{L}^{\alpha}_k(x)\big\}_{ k\ge 0}$
 and Laguerre functions $\big\{ l^{\alpha,\beta}_k(x) = e^{-\frac{\beta}{2}x}  \mathcal{L}^{\alpha}_k(\beta x)  \big\}_{k\ge 0}$ for $\beta>0$  to allow $\alpha$ being any real number.
 The resulting generalized  Laguerre  functions
 are proven to be the eigenfunctions of certain high order  Sturm-Liouville differential operators (see \refl{GLFeigen} of this paper).
 Moreover, they  are  complete and mutually orthogonal in $H^r_{x^{r+\alpha}}(\Lambda)$
 for any nonnegative
 integer $r>-\alpha-1$  with respect to an equivalent Sobolev inner product (see \refe{2.27s} of this paper).

Since the problem is dependent  on the  inner product originated from
the coercive  bilinear form of the  elliptic equation, it does  not necessarily  coincide
with the equivalent Sobolev inner product,
further  efforts should
be  paid to obtain the Fourier-like basis functions for
a fully diagonalized spectral approximation,
in spite of the Sobolev orthogonality of $\big\{ l^{\alpha,\beta}_k(x)  \big\}_{k\ge 0}$.
Starting with $\big\{ l^{-1,\beta}_k(x)  \big\}_{k\ge 0}$,
stable and efficient algorithms are then proposed  to construct the Fourier-like basis functions
for the non-homogeneous Dirichlet and Robin boundary value problems of the second order
elliptic equations.
In the sequel, both the exact solution and the approximate solution can be represented  as
 infinite and truncated Fourier series in $\big\{ l^{-1,\beta}_k(x)  \big\}$, respectively.
Although the fully diagonalized spectral methods are studied   for second order equations,
they  can be  readily generalized  to solve $2r$-th order equations
by  starting with $\big\{ l^{-r,\beta}_k(x)  \big\}$.

An ideal spectral approximation to differential equations may guarantee an optimal error estimate in  its  convergence  analysis. To match this requirement,  various orthogonal projections involving  different orders of derivatives
and boundary conditions
  have been designed and studied case by case, which frequently make  the numerical analysis
in spectral method a  tedious
task.
Moreover, the  traditional routine to measure the approximation error is first to establish the norm defined by a second-order self-adjoint differential operator, and then estimate the upper bound of the approximation error  with the induced norms. However, this practical approach usually fails to characterize
the function space in which the orthogonal projection  has an optimal error estimate.

To conquer  these difficulties, we need a unified definition
of the orthogonal spectral projections with a systematic numerical  analysis.
Fortunately, the Sobolev orthogonality of the  generalized Laguerre functions $\big\{ l^{\alpha,\beta}_k(x)  \big\}$ with
a negative integer $\alpha=-n$  enables  us to  define the unified
 orthogonal projection $\pi^{-n,\beta}_N$ from $H^{n}_{x^{r-n}}(\Lambda)$
to the finite approximation space
for all nonnegative integer $r\ge n$, ignoring the specific value of $r$.
More importantly,  such an  orthogonal projection $\pi^{-n,\beta}_N$ interpolates
the endpoint function values up to the $(n-1)$-th derivative, i.e,
$\partial_x^{\ell}\pi^{-n,\beta}_Nu(0)= \partial_x^{\ell}u(0)$ for any $0 \le\ell \le n-1 $  and  $N\ge n$.
This endpoint interpolation property ensures $u-\pi^{-n,\beta}_Nu\in H_0^n(\Lambda)$,  thus makes $\pi^{-n,\beta}_N$  applicable to  both the Dirichlet and Robin boundary value problems, and available to  multi-domain spectral methods.
Besides, owing to the clarity of the orthogonality structure of the generalised Laguerre functions,
one can not only derive an optimal order of the convergence for the approximated  function, but also get a generic  characterization of the function space where  the orthogonal projection has an optimal error estimate.

Therefore, the second purpose of this paper is to establish such a unified orthogonal Laguerre projection, and apply
it to the convergence analysis on the fully diagonalized Laguerre spectral method
for both the Dirichlet and Robin boundary value problems of second order elliptic equations.

The remainder of the paper is organized as follows. In Section 2, we first make conventions on the
frequently used notations, and then introduce  generalized Laguerre polynomials and functions
with arbitrary index $\alpha$. The fully diagonalized Laguerre spectral methods and the implementation of
algorithms
are proposed in Section 3 for the Dirichlet and Robin boundary value problems of  second order elliptic equations. Section 4 is then devoted to the convergence analysis of
the unified orthogonal projection  together with  our   Laguerre spectral methods.
Finally, numerical results are presented in Section 5 to demonstrate the effectiveness
and accuracy of
the proposed diagonalized  Laguerre spectral methods, which are in agreement with
our theoretical predictions.

\section{Generalized Laguerre polynomials and functions}
\setcounter{equation}{0} \setcounter{lmm}{0} \setcounter{thm}{0}

\subsection{Notations and preliminaries} Let $\Om=(0,\infty)$ and
${\varpi}(x)$ be a weight function which is not necessary in $L^1(\Om)$. We define
$$ L^2_{\varpi}({\Om})=\{ v~|~v  \hbox{ is measurable on $ \Om$ and~} \|v\|_{\varpi}<\infty \},$$
 with the following inner product and norm,
$$(u,v)_{\varpi}=\dint_{\Om}u(x)v(x){\varpi}(x)dx,\quad\|v\|_{\varpi}=(v,v)^{\frac{1}{2}}_{\varpi},\quad
\forall u,v\in L_{\varpi}^2(\Om). $$ For simplicity, we denote
$\frac{d^kv}{dx^k}=\partial^k_x v$ and $\frac{dv}{dx}=v^\prime.$
 For any  integer  $ m\ge 0$, we define
$$ H_{\varpi}^m(\Om)=\{ v~|~\partial_x^k v\in L^2_{\varpi}(\Om),~ 0\leq k\leq m \}, $$
with the following semi-norm  and norm,
$$ |v|_{m, {\varpi}}=\|\partial_x^m v \|_{\varpi},\qquad \|v\|_{m,{\varpi}}=\Big(\dsum_{k=0}^m
|v|_{k,{\varpi}}^2\Big)^{\frac 1 2}.$$ For any real $r>0,$ we define the
space $ H^r_{\varpi}(\Om) $ and its norm $\|v\|_{r,{\varpi}}$ by function
space interpolation as in \cite{BL}. In cases where no confusion
arises,  ${\varpi}$ may be dropped from the notations whenever
${\varpi}(x)\equiv 1.$  Specifically, we shall use  the weight functions $w=w(x)=x$  and   $w^{\alpha}=w^{\alpha}(x)=x^{\alpha}$
in the subsequent sections.

We denote by $\mathbb{R}$ the collection of real numbers,   by $\mathbb{N}_0$ and  $\mathbb{Z}^-$ the collections of
nonnegative  and negative integers, respectively. Further, we let $\mathbb{P}_k$ be the space of   polynomials of degree  $\le k$.

 Let $\aleph:=\mathbb{Z}^-\cup (-1,+\infty)$. We also  define the characteristic functions $ \chi_n$ for $n\in \mathbb{N}_0$,
\begin{align*}
\chi_n(\alpha)=\begin{cases}
-\alpha, & \alpha+n\in \mathbb{Z}^-,\\
0, &  \alpha+n\in(-1,+\infty).
\end{cases}
\end{align*}
For short we write  $ \chi(\alpha)=\chi_0(\alpha)$.

\subsection{Generalized Laguerre polynomials}
It is well known that, for $\af>-1,$ the classical Laguerre polynomials ${\mathcal{L}}_k^{{\af}}(x),\ k=0,1,\dots,$ admit an explicit representation (see \cite{Szeg75}):
\bgeq{2.1s}
{\mathcal{L}}_k^{{\af}}(x)
=\dsum_{\nu=0}^k \frac{(\alpha+\nu+1)_{k-\nu} }{(k-\nu)!\nu!} (-x)^\nu,\qquad
 x\in \Lambda,\quad k\geq 0,
 \edeq
where we use the Pochhammer symbol
$(a)_{n}= a(a+1)\dots(a+n-1)$ for any $a\in \mathbb{R}$ and $n\in \mathbb{N}_0$.

The classical Laguerre polynomials can be extended to cases with any $\alpha\in \mathbb{R}$
and the same representation as \eqref{2.1s},  which are referred to as the
generalized Laguerre polynomials (cf. \cite{PP}). Obviously, the generalized Laguerre polynomials ${\mathcal{L}}_k^{{\af}}(x),\ k=0,1,\dots,$  constitute a complete basis for the
linear space of real polynomials as well, since $\deg {\mathcal{L}}^{\af}_k =k$ for all $k\geq 0$.

The generalized Laguerre polynomials fulfill the following
recurrence relations. \bglm{lm:2.1s}For any $\af\in \mathbb{R}$, it
holds \bgeq{2.5s}\left\{
\begin{array}{ll}
{\mathcal{L}}^{\af}_{0}(x)=1,\qquad\qquad {\mathcal{L}}^{\af}_{1}(x)=-x+\af+1,\nns
(k+1){\mathcal{L}}^{\af}_{k+1}(x)=(2k+\af+1-x){\mathcal{L}}^{\af}_k(x)-
(k+\af){\mathcal{L}}^{\af}_{k-1}(x),\qquad k\ge 1.
\end{array}\right.
\edeq
\edlm
\begin{proof} The recurrence relation \refe{2.5s} for $\alpha\in\mathbb{R}$ can be derived from those of the classic
Laguerre polynomials for $\alpha>-1$ by the continuation method. Here, we also give a concrete proof by the representation
\refe{2.1s}.
Using the expression \refe{2.1s}, we obtain that for integer $k\geq 1,$
$$
\begin{array}{ll}
(2k+\af+1){\mathcal{L}}^{\af}_k(x)-(k+\af){\mathcal{L}}^{\af}_{k-1}(x)
-(k+1){\mathcal{L}}^{\af}_{k+1}(x)
\nns=(2k+\af+1)\dsum_{\nu=0}^k \frac{(-1)^\nu\Gamma(k+\af+1)}{\nu!(k-\nu)!\Gamma(\af+\nu+1)}x^{\nu}
-(k+\af)\dsum_{\nu=0}^{k-1} \frac{(-1)^\nu\Gamma(k+\af)}{\nu!(k-\nu-1)!\Gamma(\af+\nu+1)}x^{\nu}\\
\quad-(k+1)\dsum_{\nu=0}^{k+1} \frac{(-1)^\nu\Gamma(k+\af+2)}{\nu!(k-\nu+1)!\Gamma(\af+\nu+1)}x^{\nu}.
\end{array}
$$
Then a direct computation shows that
\begin{align*}
(2k+\af&+1){\mathcal{L}}^{\af}_k(x)-(k+\af){\mathcal{L}}^{\af}_{k-1}(x)
-(k+1){\mathcal{L}}^{\af}_{k+1}(x)
\\
=&\,\dsum_{\nu=1}^{k+1} \dfrac{(-1)^{\nu-1}\Gamma(k+\af+1)}{(\nu-1)!(k-\nu+1)!\Gamma(\af+\nu)}x^\nu
= x {\mathcal{L}}^{\af}_{k}(x).
\end{align*}
The desired result is now derived.
\end{proof}

\bglm{lm:2.2s}For any $\af\in \mathbb{R}$ and $k\geq 0$, it holds
\begin{align}
\label{2.6s}
&{\mathcal{L}}^{\af}_k(x)={\mathcal{L}}^{\af+1}_k(x)-{\mathcal{L}}^{\af+1}_{k-1}(x),
\\
\label{2.7s}
&\partial_{x} {\mathcal{L}}^{\af}_k(x)=-{\mathcal{L}}^{\af+1}_{k-1}(x),
\\
\label{2.8s}
&x\partial_{x} {\mathcal{L}}^{\af}_k(x)=k{\mathcal{L}}^{\af}_{k}(x)-(k+\af){\mathcal{L}}^{\af}_{k-1}(x),
\\
\label{2.9s}
&{\mathcal{L}}^{\af}_k(x)=\partial_{x} {\mathcal{L}}^{\af}_k(x)-\partial_{x} {\mathcal{L}}^{\af}_{k+1}(x),
\end{align}
where ${\mathcal{L}}^{\af}_{k}(x)\equiv0$ for any $k\in
\mathbb{Z}^{-}.$ \edlm
\begin{proof} The recurrence relations \refe{2.6s}-\refe{2.8s} can be obtained readily by using similar arguments as in \refl{lm:2.1s}. Moreover, by \refe{2.6s} and \refe{2.7s}, it is easy to derive \refe{2.9s}.
\end{proof}

\bglm{lm:2.3s}
For any $\af\in \mathbb{R}$, the generalized Laguerre polynomials ${\mathcal{L}}^{\af}_{k}(x)$ satisfy the Sturm-Liouville equation
\bgeq{2.12s}
  x^{-\alpha}e^{x}\partial_{x}(x^{\af+1}e^{-x}\partial_{x}\mathcal{L}_{k}^{\af}(x))+\lambda_{k}\mathcal{L}_{k}^{\af}(x)=0,\quad k\geq 0,
\edeq
or equivalently,
\bgeq{2.13s}
x\partial_{x}^{2}\mathcal{L}_{k}^{\af}(x)+(\af+1-x)\partial_{x}\mathcal{L}_{k}^{\af}(x)
+\lambda_{k}\mathcal{L}_{k}^{\af}(x)=0,\quad k\geq 0,
\edeq
with the corresponding eigenvalue $\lbd_k=k$.
\edlm
\begin{proof}
Lemma \ref{lm:2.3s} can be proved  by  the continuation method from the Sturm-Liouville equation of the classic
Laguerre polynomials for $\alpha>-1$. Also one can  give a proof by  using the representation
\refe{2.1s}. We omit the details here.
\end{proof}

We are interested in those generalized Laguerre polynomials with an integer index
$\alpha\in \mathbb{Z}^-$.
\bglm{lm:2.4s}
For any $\af\in\mathbb{Z}^-$,  we have
\bgeq{2.15s}
\mathcal{L}_{k}^{\af}(x)=(-x)^{-\af}\frac{(k+\af)!}{k!}\mathcal{L}_{k+\af}^{-\af}(x),\qquad k\ge \chi(\alpha).
\edeq
And for  any $\alpha\in \aleph$,  the following orthogonality relation holds:
\bgeq{2.16s}
  \dint_{\Lambda}{\mathcal{L}}^{\af}_k(x)
  {\mathcal{L}}^{\af}_m(x)x^{\af}e^{-x}dx =\gm_k^{\af}\delta_{k,m}, \quad \gm_k^{\af}=\frac{\Gamma(k+\af+1)}{k!} ,\qquad k,m \ge  \chi(\alpha),\edeq
  where $\dt_{k,m}$ is the Kronecker symbol.
\edlm
\begin{proof} The identity \refe{2.15s} comes directly from \cite{Szeg75}. The  orthogonality  relation \refe{2.16s}
is known  for classic Laguerre polynomials with  $\alpha\in (-1,+\infty)$;  while  for  $\alpha\in \mathbb{Z}^-$,
\eqref{2.16s}
can be obtained immediately from \refe{2.15s} together with \eqref{2.16s} for $\alpha\in (-1,+\infty)$.
\end{proof}

We now conclude this subsection with some generalized Laguerre polynomials ${\mathcal{L}}^{\af}_k(x)$ for $\alpha\in \mathbb{Z}^-$.
{\renewcommand{\arraystretch}{1.35}
\addtolength{\arraycolsep}{0.4em}
\begin{align*}
\begin{array}{ccccccc}
                         &  k=0 &   k=1       &  k=2                       &  k=3                                  &  \dots    & k\ge \chi(\alpha)   \\ \hline
 \alpha=-1     &  1     &   -x          &  \tfrac12 x(x-2)       &  -\tfrac16  x (x^2-6x+ 6)    &  \dots    &  -\frac{1}{k} x  {\mathcal{L}}^{1}_{k-1}(x)     \\ \hline
 \alpha=-2     &   1    &   -x-1      &    \tfrac12 x^2         &   -\tfrac16 x^2 (x-3)            &  \dots    &  \frac{1}{k(k-1)} x^2 {\mathcal{L}}^{2}_{k-2}(x)            \\ \hline
 \alpha=-3     &  1    &     -x-2     &   \tfrac12 x^2+x+1  &   -\tfrac16 x^3                    &  \dots    &  -\frac{1}{k(k-1)(k-2)}  x^3 {\mathcal{L}}^{3}_{k-3}(x)           \\ \hline
 \dots            &  \dots  &   \dots  &  \dots                      &   \dots                                &  \dots    &  \dots
\end{array}
\end{align*}
\renewcommand{\arraystretch}{1}
\addtolength{\arraycolsep}{-0.4em}
}

\subsection{Generalized Laguerre functions} In this subsection, we shall introduce the generalized Laguerre functions with arbitrary parameters $\af\in\mathbb{R}$ and $\bt>0$
and present some properties.

The generalized Laguerre
functions $l_k^{\alpha,\beta}(x),$ $k\ge 0$ are defined by
\bgeq{2.17s}l_k^{\alpha,\beta}(x)=e^{-\frac{1}{2}{\beta}x}\mathcal{L}_k^{\alpha}(\beta x),\qquad
\forall\af\in \mathbb{R},\quad\bt>0,\edeq
and the multiplication of
$e^{-\frac12\beta x}$ and the leading term of $ \mathcal{L}_k^{\alpha}(\beta x)$
is simply referred to as the leading term of $l^{\alpha,\beta}_k(x)$.

According to \refe{2.15s}, for any $\af\in\mathbb{Z}^{-}$, we have
\bgeq{2.18s}
l_{k}^{\af,\bt}(x)=(-\bt x)^{-\af}\frac{(k+\af)!}{k!}l_{k+\af}^{-\af,\bt}(x), \quad k\ge \chi(\alpha),
\edeq
which means that   $x=0$  is a zero  of $l_k^{\alpha,\beta}(x)$  with the multiplicity $-\alpha$, i.e.,
\begin{align}
\label{homoBnd}
\partial_x^{\nu}l_k^{\alpha,\beta}(0)=0,\quad k\ge \chi(\alpha), \quad \nu=0,1,\dots,-\alpha-1.
\end{align}

Due to \refe{2.5s}-\refe{2.9s},
the generalized Laguerre functions satisfy the following recurrence
relations:
\begin{lmm}
For any $k\in \mathbb{N}_{0}$,  it holds that
\begin{align}
\label{2.21s}
& {\beta} x  l_k^{\alpha,\beta}(x) =
- (k+1)l_{k+1}^{\alpha,\beta}(x)+(2k+\alpha+1)l_k^{\alpha,\beta}(x)-(k+\alpha)l_{k-1}^{\alpha,\beta}(x),
\\
 \label{2.22s}
&l^{\af,\bt}_k(x)=l^{\af+1,\bt}_k(x)-l^{\af+1,\bt}_{k-1}(x),
\\
\label{2.23s}
&\partial_x l_k^{\alpha,\beta}(x)=-\beta l_{k-1}^{\alpha+1,\beta}(x)-\frac{\beta}{2} l_k^{\alpha,\beta}(x)
= -\frac{\beta}{2}\big[ l_k^{\alpha+1,\beta}(x) + l_{k-1}^{\alpha+1,\beta}(x)\big],
\\
\label{2.24s}
\begin{split}
        &x\partial_x l_k^{\alpha,\beta}(x)
                             =\frac{k+1}{2} l_{k+1}^{\alpha,\beta}(x) - \frac{\alpha+1}{2}l_k^{\alpha,\beta}(x)
                       - \frac{k+\alpha}{2} l_{k-1}^{\alpha,\beta}(x),
\end{split}\\
\label{2.25s}
\begin{split}
\partial_x l_k^{\alpha,\beta}(x)-\partial_x l_{k+1}^{\alpha,\beta}(x)=\frac{1}{2}\beta(l_k^{\alpha,\beta}(x)+l_{k+1}^{\alpha,\beta}(x)).
\end{split}
 \end{align}
 Hereafter, we use the convention that $l_k^{\alpha,\beta}(x)=0$  whenever $k\in \mathbb{Z}^{-}$.
 \end{lmm}

The generalized Laguerre functions are eigenfunctions of certain singular Sturm-Liouville differential  operators.
\begin{lmm} For any $n\in \mathbb{N}_0$, it holds that
\label{GLFeigen}
\begin{align}
\label{2.20s}
  \sum_{\nu=0}^{n}(-1)^{\nu} \binom{n}{\nu}  \frac{\beta^{2n-2\nu}}{2^{2n-2\nu}} x^{-\alpha} \partial_x^{\nu}
  \Big( x^{\alpha+n} \partial_x^{\nu} l^{\alpha,\beta}_k  \Big) =
 \frac{\beta^{n}}{2^{n}}     \lambda_{k,n}^{\alpha}  l^{\alpha,\beta}_k,
  \quad k\ge 0,
\end{align}
where  $ \lambda^{\alpha}_{k,n}$ satisfies the following recurrence relation,
 \begin{align}
 \label{eigenInduct}
 \lambda^{\alpha}_{k,0}=1,\qquad
 \lambda^{\alpha}_{k,n}= (k+\alpha+1)  \lambda^{\alpha+1}_{k,n-1}
 +k  \lambda^{\alpha+1}_{k-1,n-1}, \quad n\ge 1, \; k\ge 0.
 \end{align}
 \end{lmm}
\begin{proof} We prove \eqref{2.20s} and \eqref{eigenInduct} by induction.
It is obvious that \eqref{2.20s} holds for $n=0$. Moreover, by virtue of \refe{2.12s} and \refe{2.17s} we have
\begin{align*}
-x^{-\alpha} \partial_{x} \Big(  x^{\alpha+1} \partial_{x}l_k^{\alpha,\beta}(x)\Big)
+\frac{\beta^2}{4}x l_k^{\alpha,\beta}(x)= \frac{\beta}{2} (2k+\alpha+1) l_k^{\alpha,\beta}(x)
=  \frac{\beta}{2} \lambda^{\alpha}_{k,1} l_k^{\alpha,\beta}(x),\quad k\geq 0,
\end{align*}
which gives \eqref{2.20s} and \eqref{eigenInduct} for $n=1$.

We now assume that  \eqref{2.20s} and \eqref{eigenInduct}  hold for an integer $n\ge 1$. Then
by the  recursive formula of  binomial coefficients
together with \eqref{2.22s} and \eqref{2.23s},
\begin{align*}
I:=&\sum_{\nu=0}^{n+1}(-1)^{\nu} \binom{n+1}{\nu}   \frac{\beta^{2n+2-2\nu}}{2^{2n+2-2\nu}}  x^{-\alpha} \partial_x^{\nu}
  \Big( x^{n+1+\alpha} \partial_x^{\nu} l^{\alpha,\beta}_k  \Big)
\\
 =&\sum_{\nu=0}^{n+1}(-1)^{\nu} \left[ \binom{n}{\nu}+ \binom{n}{\nu-1}\right]   \frac{\beta^{2n+2-2\nu}}{2^{2n+2-2\nu}}  x^{-\alpha} \partial_x^{\nu}
  \Big( x^{n+1+\alpha} \partial_x^{\nu} l^{\alpha,\beta}_k \Big)
\\
 =& \frac{\beta^2}{4} x \sum_{\nu=0}^{n}(-1)^{\nu} \binom{n}{\nu}   \frac{\beta^{2n-2\nu}}{2^{2n-2\nu}}   x^{-(\alpha+1)} \partial_x^{\nu}
  \Big[ x^{n+(\alpha+1)} \partial_x^{\nu}\Big(  l^{\alpha+1,\beta}_k -  l^{\alpha+1,\beta}_{k-1} \Big) \Big]
  \\
& +\frac{\beta}{2} x^{-\alpha} \partial_x \sum_{\nu=1}^{n+1}(-1)^{\nu-1} \binom{n}{\nu-1}   \frac{\beta^{2n-2(\nu-1) }}{2^{2n-2(\nu-1)} } \partial_x^{\nu-1}
  \Big[ x^{n+(\alpha+1)} \partial_x^{\nu-1} \left(  l^{\alpha+1,\beta}_k + l^{\alpha+1,\beta}_{k-1} \right)  \Big].
\end{align*}
Thus by the induction assumption, \eqref{2.21s}, \eqref{2.24s} and \eqref{2.22s}, we derive that
\begin{align*}
 I=&    \frac{\beta^{n+2}}{2^{n+2}}  x \left[\lambda^{\alpha+1}_{k,n} l^{\alpha+1,\beta}_{k} - \lambda^{\alpha+1}_{k-1,n} l^{\alpha+1,\beta}_{k-1} \right]
+  \frac{\beta^{n+1}}{2^{n+1}}  x^{-\alpha}\partial_x \left[ x^{\alpha+1}\left( \lambda^{\alpha+1}_{k,n} l^{\alpha+1,\beta}_{k} +  \lambda^{\alpha+1}_{k-1,n} l^{\alpha+1,\beta}_{k-1}\right) \right]
\\
=&\frac{\beta^{n+1}}{2^{n+1}}  \lambda^{\alpha+1}_{k,n} \left[ (\alpha+1+ \tfrac{\beta}{2} x ) l^{\alpha+1,\beta}_{k}
+ x \partial_x l^{\alpha+1,\beta}_{k}\right]
+ \frac{\beta^{n+1}}{2^{n+1}} \lambda^{\alpha+1}_{k-1,n} \left[ ( \alpha+1-\tfrac{\beta}{2} x ) l^{\alpha+1,\beta}_{k-1}
+ x \partial_x l^{\alpha+1,\beta}_{k-1}\right]
\\
=& \frac{\beta^{n+1}}{2^{n+1}}  \lambda^{\alpha+1}_{k,n} (k+\alpha+1)  \big[ l^{\alpha+1,\beta}_{k}-  l^{\alpha+1,\beta}_{k-1}\big]
+ \frac{\beta^{n+1}}{2^{n+1}}  \lambda^{\alpha+1}_{k-1,n} k  \big[ l^{\alpha+1,\beta}_{k}-  l^{\alpha+1,\beta}_{k-1}\big]
\\
=&\frac{\beta^{n+1}}{2^{n+1}}  \big[ (k+\alpha+1)  \lambda^{\alpha+1}_{k,n} +  k\lambda^{\alpha+1}_{k-1,n}\big]   l^{\alpha,\beta}_{k},
\end{align*}
which is exactly  \eqref{2.20s} and \eqref{eigenInduct} with $n+1$ in place of $n$.
This ends the proof.
\end{proof}

For any $n\in \mathbb{N}_0$, $\alpha\in \mathbb{R}$ and  $\beta>0$,  define the bilinear form on $H^n_{w^{\alpha+n}}(\Lambda)\times H^n_{w^{\alpha+n}}(\Lambda)$,
\begin{align}
\label{innProd}
a^{\alpha,\beta}_{n}(u,v) = \sum_{\nu=0}^n  \binom{n}{\nu} \frac{\beta^{2n-2\nu} }{2^{2n-2\nu} } (\partial_x^{\nu} u,
 \partial_x^{\nu} v)_{w^{\af+n}}.
\end{align}
It is obvious that $a^{\alpha,\beta}_{n}(\cdot,\cdot) $ is an inner product on $H^n_{w^{\alpha+n}}(\Lambda)$
if  $\alpha+n>-1$.
 \begin{thm}
 \label{lm:2.5s}
The generalized Laguerre functions $l_k^{\alpha,\beta}(x), \, k\ge \chi(\alpha) $ for $\alpha\in \aleph$ are mutually orthogonal with respect to the weight function $w{^\af}$,
 \bgeq{2.26s}
 a^{\alpha,\beta}_{0}(l^{\af,\bt}_k,l^{\af,\bt}_m) =
(l^{\af,\bt}_k,  l^{\af,\bt}_m)_{w^{\af}} =\bt^{-\af-1} \gm_k^{\af}\delta_{k,m}, \qquad k,m\ge \chi(\alpha).\edeq
More generally, for any $n\in \mathbb{N}_0$ and $\alpha+n\in \aleph$,
 \begin{align}
 \label{2.27s}
a^{\alpha,\beta}_{n}(l^{\af,\bt}_k,l^{\af,\bt}_m) =\beta^{n-\alpha-1} \gamma^{\alpha}_{k,n}\delta_{k,m}, \qquad k,m\ge \chi_n(\alpha),
\end{align}
where the positive numbers  $\gamma^{\alpha}_{k,n}$  satisfy the recurrence relation
 \begin{align}
  \label{2.28s}
   \gamma^{\alpha+n}_{k,0}=\gamma^{\alpha+n}_{k},\qquad
 \gamma^{\alpha}_{k,n}
 = \frac{1}{2} \big[ \gamma^{\alpha+1}_{k,n-1}+ \gamma^{\alpha+1}_{k-1,n-1}  \big], \qquad n\ge 1,
\end{align}
under the convention that $\gamma^{\alpha}_{k,n}=0$ whenever $k\in \mathbb{Z}^-$.
\end{thm}
\begin{proof}
The orthogonality \eqref{2.26s}  is an immediate consequence of \eqref{2.16s}.
Meanwhile, the  recursive formula \refe{innProd} of  binomial coefficients together with
\eqref{2.22s} and \eqref{2.23s} yields
  \begin{align}\label{2.28z}
a^{\alpha,\beta}_{n+1}&\big(l^{\af,\bt}_k,l^{\af,\bt}_m\big)
=  \sum_{\nu=0}^{n+1}  \Big[\binom{n}{\nu-1}+ \binom{n}{\nu} \Big]  \frac{\beta^{2n+2-2\nu} }{2^{2n+2-2\nu} }\big( \partial_x^{\nu} l^{\af,\bt}_k,
 \partial_x^{\nu} l^{\af,\bt}_m\big)_{w^{\af+n+1}}\nonumber
 \\
=&\,
 \frac{\beta^2}{4}  \sum_{\nu=1}^{n+1}  \binom{n}{\nu-1}  \frac{\beta^{2n-2(\nu-1)}}{2^{2n-2(\nu-1)}} \big( \partial_x^{\nu-1} \big[l^{\af+1,\bt}_k +l^{\af+1,\bt}_{k-1}  \big],
 \partial_x^{\nu-1}  \big[l^{\af+1,\bt}_m  +l^{\af+1,\bt}_{m-1}  \big]\big)_{w^{n+(\alpha+1)}}\nonumber
 \\
 &\, +
  \frac{\beta^2}{4}  \sum_{\nu=0}^n  \binom{n}{\nu}  \frac{\beta^{2n-2\nu}}{2^{2n-2\nu}} \big (\partial_x^{\nu}
 \big[ l^{\af+1,\bt}_k-l^{\af+1,\bt}_{k-1}\big],
 \partial_x^{\nu}  \big[ l^{\af+1,\bt}_m-l^{\af+1,\bt}_{m-1}\big] \big)_{w^{n+(\alpha+1)}}
 \\
 =&\, \frac{\beta^2}{2}  \sum_{\nu=0}^n  \binom{n}{\nu}  \frac{\beta^{2n-2\nu}}{2^{2n-2\nu}}
 \big[ \big(\partial_x^{\nu}  l^{\af+1,\bt}_k, \partial_x^{\nu}   l^{\af+1,\bt}_m\big)_{w^{n+(\alpha+1)}}
  + \big(\partial_x^{\nu} l^{\af+1,\bt}_{k-1}, \partial_x^{\nu} l^{\af+1,\bt}_{m-1}\big)_{w^{n+(\alpha+1)}}  \big]
  \nonumber\\
 =&\, \frac{\beta^2}{2} \big[ a^{\alpha+1,\beta}_{n}\big(l^{\alpha+1,\bt}_k,l^{\alpha+1,\bt}_m\big) + a^{\alpha+1,\beta}_{n}\big(l^{\alpha+1,\bt}_{k-1},l^{\alpha+1,\bt}_{m-1}\big)  \big].\nonumber
  \end{align}

To complete the proof of  \refe{2.27s}, we proceed by induction on $n.$
 By \refe{2.28z} we get
    \begin{align}\label{2.29z}a^{\alpha,\beta}_{1}\big(l^{\af,\bt}_k,l^{\af,\bt}_m\big)= \frac{\beta^2}{2} \big[  \beta^{-(\alpha+1)-1}
    \gamma^{\alpha+1}_{k,0}  + \beta^{-(\alpha+1)-1}
    \gamma^{\alpha+1}_{k-1,0}   \big]\delta_{k,m}
   =   \frac{\beta^{-\alpha} }{2} \big[
    \gamma^{\alpha+1}_{k,0}  +
    \gamma^{\alpha+1}_{k-1,0}   \big]\delta_{k,m},
    \end{align}
if either  (a). $\alpha+1\in \{-1,-2,-3,\cdots\} $ and $k,m\geq -\alpha$; or (b). $\alpha+1\in (-1,+\infty)$ and $k,m\geq 0. $  This exactly gives
\eqref{2.27s} for $k,m\ge \chi_{1}(\alpha)$ with $n=1$.

Assume that the result \refe{2.27s} for $k,m\ge \chi_{n}(\alpha)$ with $n=p$ holds. We now verify the result with $n=p+1.$ Clearly, by \refe{2.28z} we have
    \begin{align*}a^{\alpha,\beta}_{p+1}&\big(l^{\af,\bt}_k,l^{\af,\bt}_m\big)= \frac{\beta^2}{2} \big[  \beta^{p-(\alpha+1)-1}
    \gamma^{\alpha+1}_{k,p}  + \beta^{p-(\alpha+1)-1}
    \gamma^{\alpha+1}_{k-1,p}   \big]\delta_{k,m}
   =    \frac{\beta^{p-\alpha}}{2} \big[
    \gamma^{\alpha+1}_{k,p}  +
    \gamma^{\alpha+1}_{k-1,p}   \big]\delta_{k,m},\end{align*}
    if   either (a). $\alpha+1 \in \{ -p-1,-p-2,\cdots  \} $ and  $k,m\ge  -\alpha  $;  or (b). $\alpha+1>-p-1 $ and   $k,m\ge 0$.
    This statement implies the result \refe{2.27s} for $k,m\ge \chi_{n+1}(\alpha)$ with $n=p+1.$ This ends the proof.
\end{proof}

The normalization constants $\gamma^{\alpha}_{k,n}$ and the eigenvalues  $\lambda^{\alpha}_{k,n}$ are closely related. In effect, for any $\alpha\in \aleph$, $  k\ge \chi(\alpha),\, n\in \mathbb{N}_0$,
we get  that
\begin{align}
\label{2.29s}
\begin{split}
\gamma^{\alpha}_{k,n} \overset{\eqref{2.27s}}=&\, \beta^{\alpha+1-n}  a^{\alpha,\beta}_n(l^{\alpha,\beta}_k, l^{\alpha,\beta}_k)
\overset{\eqref{innProd}} =  \beta^{\alpha+1-n} \sum_{\nu=0}^{n} \binom{n}{\nu} \frac{\beta^{2n-2\nu}}{2^{2n-2\nu}}
\big(\partial_x^{\nu}l^{\alpha,\beta}_k, \partial_x^{\nu}l^{\alpha,\beta}_k\big)_{w^{\alpha+n}}
\\
=&\,\beta^{\alpha+1-n}   \sum_{\nu=0}^{n} (-1)^{\nu}\binom{n}{\nu} \frac{\beta^{2n-2\nu}}{2^{2n-2\nu}}
\big(  \partial_x^{\nu}\big[ w^{\alpha+n} \partial_x^{\nu}l^{\alpha,\beta}_k\big], l^{\alpha,\beta}_k\big)
\\
\overset{\eqref{2.20s}}=&\,\beta^{\alpha+1-n}  \frac{\beta^n}{2^n}  \lambda^{\alpha}_{k,n}
\big( l^{\alpha,\beta}_k,  l^{\alpha,\beta}_k\big)_{w^{\alpha}}
  \overset{\eqref{2.26s}}=\frac{1}{2^n}  \lambda^{\alpha}_{k,n} \gamma^{\alpha}_k,
 \end{split}
\end{align}
where the third inequality sign is obtained by integration by parts combined with \eqref{homoBnd}.

Moreover,  for sufficiently large $k$, an induction procedure starting with \eqref{eigenInduct} reveals
   \begin{align}
   \label{eigenEsti}
    \lambda^{\alpha}_{k,n}= (2k+\alpha+1)^n + \mathcal{O}(k^{n-2}),
   \end{align}
   which  implies
   \begin{align}
   \label{ratio}
 \frac{\gamma^{\alpha}_{k,n}}{\gamma^{\alpha}_{k,n+1}}  = \frac{2}{2k+\alpha+1}  + \mathcal{O}(k^{-3}).
 \end{align}
 The following eigenvalues $\lambda^{-n}_{k,n}$ and normalization constants $\gamma^{-n}_{k,n}$  are of our particular interest,
 \begin{align}
  \label{eigenExplcit}
  &\lambda^{0}_{k,0} = 1,\qquad \qquad\quad
\lambda^{-n}_{k,n} = 2^n (k-n+1)_n, \qquad k,n \in \mathbb{N}_0,
\\
\label{normExplicit}
& \gamma^{0}_{k,0} =1, \qquad \qquad\quad  \gamma^{-n}_{k,n}  = \frac{1}{2^n} \sum_{\nu=0}^{\min(k,n)} \binom{n}{\nu} ,\qquad      k,n \in \mathbb{N}_0.
 \end{align}

\section{Fully diagonalized spectral methods}\setcounter{equation}{0} \setcounter{lmm}{0} \setcounter{thm}{0}
In this section, we propose the fully diagonalized spectral methods
using generalized Laguerre functions for solving differential equations
on the half line. The main idea is to find a system of Sobolev orthogonal
functions \cite{FerMarPinXu,MarXu} with respect to the coercive bilinear form arising from differential equation,
such that both the exact solution and the approximate solution can be explicitly expressed
as a Fourier series in the Sobolev orthogonal
functions. Although we only consider in this section
non-homogenous  Robin/Drichlet boundary  value problems of a second order equation,
one can extend the  fully diagonalized spectral methods for solving partial  differential equations
of an arbitrary high order.

\subsection{Robin boundary value problems}\setcounter{equation}{0} \setcounter{lmm}{0}
Consider the second order elliptic boundary value problem:
\bgeq{5.1} \left\{\begin{array}{ll}-u^{\prime\prime}(x)+\gamma
u(x)=f(x),\quad \gamma\ge 0,\quad x\in\Lambda,\nns
-u^{\prime}(0)+\mu u(0)=\eta,\quad\displaystyle\lim_{{x\rightarrow+\infty}}u(x)=0,\quad \mu\geq 0.\end{array}\right.
\edeq
A weak formulation of \refe{5.1} is to find $u\in
H^{1}(\Lambda)$ such that \bgeq{5.2}
A_{\gamma,\mu}(u,v):=\mu u(0)v(0)+(u^\prime,v^\prime)+\gamma
(u,v)=(f,v)+\eta\,v(0),\quad \forall v\in H^{1}(\Lambda). \edeq
The Lax-Milgram lemma guarantees a unique solution to \eqref{5.2}
if $f\in (H^{1}(\Lambda))^{\prime}$.

Let
 $$   X_N^{\bt}:=\{e^{-\frac12 \bt x} p(x): p\in \mathbb{P}_N \}
 =   \{ l^{-1,\beta}_k: 0\le k\le N  \}.$$
The generalized Laguerre spectral scheme for \refe{5.1} is to find $u_{N}\in
X^{\bt}_{N}$, such that
\bgeq{5.3}
A_{\gamma,\mu}(u_N,v_N)=(f,v_{N})+\eta v_N(0),\quad
\forall v_{N}\in X^{\bt}_{N}.
\edeq

For an efficient approximation scheme,  one usually chooses the generalized Laguerre functions
$\{l_{k}^{-1,\beta}(x)\}_{0\leq k\leq N}$ as the basis functions for problem \refe{5.3}.
However, we are eager for  an ideal approximation scheme whose (total) stiff matrix,  in analogue to the Fourier spectral method for periodic problem, is diagonal. Obviously,  the utilization of the basis functions $\{l_{k}^{-1,\beta}(x)\}_{0\leq k\leq N}$
leads to a tridiagonal algebraic system. 
To this end,  we shall construct  new basis functions $\{\mathcal{R}_{k}^{\bt}(x)\}_{0\leq k\leq N}$
which are mutually orthogonal with respect to the Sobolev  inner product $A_{\gamma,\mu}(\cdot,\cdot)$ instead of $a_1^{-1,\beta}(\cdot,\cdot)$ defined in \reft{lm:2.5s}.

\bglm{lm:5.1}  Let  $\mathcal{R}_{k}^{\bt}\in X^{\bt}_k,\, k\in \mathbb{N}_0$ be the Sobolev orthogonal Laguerre functions such that
 $\mathcal{R}_{k}^{\bt}-l^{-1,\beta}_{k}\in X^{\bt}_{k-1}$  and
\begin{equation}
\label{5.4}
A_{\gamma,\mu}(\mathcal{R}_{k}^{\bt},\mathcal{R}_{m}^{\bt})=\rho_{k}\delta_{k,m}, \qquad k,m\in \mathbb{N}_0.
\end{equation}
Then $ \mathcal{R}_{k}^{\bt}(x), \, k\in \mathbb{N}_0$  satisfy the following recurrence relation,
\begin{equation}\label{5.5}\mathcal{R}_{0}^{\bt}(x)=l_0^{-1,\beta}(x),\quad \mathcal{R}_{k}^{\bt}(x)=l_k^{-1,\beta}(x)-d_{k-1}\mathcal{R}_{k-1}^{\bt}(x), \quad \forall\ k\ge 1,\end{equation}
where $\rho_{0}=\mu+ \dfrac{\beta}{4}+\dfrac{\gamma}{\beta}$ and
 \begin{align*}
& d_{k-1}=\frac{\beta}{4\rho_{k-1}}-\frac{\gamma}{\beta \rho_{k-1}}, \quad
 \rho_{k}= -d^2_{k-1} \rho_{k-1} +  \frac{\beta}{2}+\frac{2\gamma}{\beta}, \qquad k \ge 1.
\end{align*}
\edlm
\begin{proof}  By the orthogonality assumption \eqref{5.4} of $\{\mathcal{R}_{k}^{\bt}\}$,
\begin{align*}
l_k^{-1,\beta}(x)=\mathcal{R}_{k}^{\bt}(x) + \dsum_{m=0}^{k-1}\frac{A_{\gamma,\mu}(l_k^{-1,\beta},  \mathcal{R}_{m}^{\bt})}{\rho_{m}} \mathcal{R}_{m}^{\bt}(x).
\end{align*}
Meanwhile, by \refe{5.2} and \refe{innProd}, for any $k>m\ge 0$,
\begin{align*}
A_{\gamma,\mu}(l_k^{-1,\beta},  \mathcal{R}_{m}^{\bt})
= a^{-1,\beta}_1(l_k^{-1,\beta},  \mathcal{R}_{m}^{\bt}) + \mu l_k^{-1,\beta}(0) \mathcal{R}_{m}^{\bt}(0)+ \big(\gamma-\frac{\beta^2}{4}\big) (l_k^{-1,\beta},  \mathcal{R}_{m}^{\bt}) .
\end{align*}
Both the first  and  the second terms in the righthand side above  are zero due to the orthogonality relation \eqref{2.27s}
 of  $\{l_k^{-1,\beta}\}$
and the homogeneity boundary condition \eqref{homoBnd} for  $l_k^{-1,\beta},\, k\ge 1$. Further by \eqref{2.22s} and the orthogonality
relation \eqref{2.26s} for  $\{l^{0,\beta}_k\}$,
\begin{align*}
A_{\gamma,\mu}(l_k^{-1,\beta},  \mathcal{R}_{m}^{\bt}) =&  \big(\gamma-\frac{\beta^2}{4}\big) (l_k^{-1,\beta},  \mathcal{R}_{m}^{\bt})=  \big(\gamma-\frac{\beta^2}{4}\big) (l_k^{-1,\beta},  l_{k-1}^{-1,\beta})\delta_{m,k-1}
\\= &\, \big(\gamma-\frac{\beta^2}{4}\big) (l_k^{0,\beta}-l_{k-1}^{0,\beta}, l_{k-1}^{0,\beta}-l_{k-2}^{0,\beta})\delta_{m,k-1}
\\=&\,
   \big(\gamma-\frac{\beta^2}{4}\big)   (-l_{k-1}^{0,\beta},  l_{k-1}^{0,\beta})\delta_{m,k-1}
    =\,\big(\frac{\beta}{4}-\frac{\gamma}{\beta}\big)   \delta_{m,k-1},  \quad k>m\ge 0,
\end{align*}
which, in return,  implies
\begin{align*}
l_k^{-1,\beta}(x)= \mathcal{R}_{k}^{\bt}(x) +  d_{k-1} \mathcal{R}_{k-1}^{\bt}(x),\quad
d_{k-1}  =   \frac{\beta}{4\rho_{k-1}}-\frac{\gamma}{\beta \rho_{k-1}},
\qquad k\ge 1.
\end{align*}

We now turn to the proof  of  the recurrence identity for  $\rho_k,k\ge 0$.
Firstly, a direct computation shows
\begin{align*}
  \rho_{0} = &\, A_{\gamma,\mu}(\mathcal{R}_0^{\beta}, \mathcal{R}_0^{\beta}) =  A_{\gamma,\mu}(l_0^{-1,\beta}, l_0^{-1,\beta})=
 \mu+ \frac{\beta}{4}+\frac{\gamma}{\beta}.
\end{align*}
Further, for $k\ge 1$,
\begin{align*}
  \beta = &\, a^{-1,\beta}_1(l_k^{-1,\beta}, l_k^{-1,\beta}) = A_{\gamma,\mu}(l_k^{-1,\beta}, l_k^{-1,\beta})
  + \big(\frac{\beta^2}{4}-\gamma\big) (l_k^{-1,\beta}, l_k^{-1,\beta})
  \\
  =&\, A_{\gamma,\mu}(\mathcal{R}_{k}^{\bt}(x) +  d_{k-1} \mathcal{R}_{k-1}^{\bt}(x), \mathcal{R}_{k}^{\bt}(x) +  d_{k-1} \mathcal{R}_{k-1}^{\bt}(x))
  + \big(\frac{\beta^2}{4}-\gamma\big) (l_k^{0,\beta}-l_{k-1}^{0,\beta},l_k^{0,\beta}-l_{k-1}^{0,\beta})
  \\
  =&\, \rho_k+d^2_{k-1} \rho_{k-1} +  \big(\frac{\beta^2}{4}-\gamma\big)  \frac{2}{\beta},
\end{align*}
where we have used  \eqref{2.27s} and  \eqref{2.28s} for the first equality sign,
\eqref{5.5} and \eqref{2.22s} for the third equality sign,
and \eqref{5.4} and  \eqref{2.26s} for the fourth equality sign.
The proof is completed.
\end{proof}

Obviously, 
$X^{\beta}_N=\{ \mathcal{R}^{\beta}_k :   0\le k\le N \}$.
 Thus the variational forms \eqref{5.2} and \eqref{5.3} together with
the orthogonality of $\{\mathcal{R}_{k}^{\bt}\}$
lead to the following main theorem in this subsection.

\begin{thm} Let $u$ and $u_N$ be the solution  of  \eqref{5.1}  and  \eqref{5.3}, respectively. Then
both $u$ and $u_N$ have the explicit representations in $\{\mathcal{R}^{\beta}_k\}$,
\begin{align*}
    & u(x) = \sum_{k=0}^{\infty} \hat u_k \mathcal{R}_k^{\beta}(x),  \qquad
    u_N(x) = \sum_{k=0}^{N} \hat u_k \mathcal{R}_k^{\beta}(x),
    \\
    &\hat u_k
    = \frac{1}{\rho_k} A_{\gamma,\mu}(u, \mathcal{R}_k^{\beta})
    = \frac{1}{\rho_k}  \left[ (f, \mathcal{R}_k^{\beta}) + \eta \mathcal{R}_k^{\beta}(0)\right], \quad k\ge 0.
\end{align*}
\end{thm}

\subsection{Dirichlet boundary value problems}

Consider the second order elliptic boundary value problem:
\bgeq{4.1s} \left\{\begin{array}{ll}-u^{\prime\prime}(x)+\gamma
u(x)=f(x),\quad &\gamma>0,\quad x\in\Lambda,\nns
u(0)=\eta,\quad\displaystyle\lim_{{x\rightarrow+\infty}}u(x)=0.\end{array}\right.
\edeq
A weak formulation of \refe{4.1s} is to find $u\in
H^{1}(\Lambda)$ such that $u(0)=\eta$ and
\bgeq{4.2s}
A_{\gamma}(u,v):= (u^\prime,v^\prime)+\gamma
(u,v)=(f,v),\quad \forall v\in H^{1}_{0}(\Lambda). \edeq
Clearly, if
$f\in(H^1_0(\Lambda))^{\prime}$, then by Lax-Milgram lemma,
\refe{4.2s} admits a unique solution.

Let
 $$   X_N^{0,\bt}:=\{e^{-\frac12 \bt x} p(x): p(0)=0 \text{ and }  p\in \mathbb{P}_N \}
 = \{ l^{-1,\beta}_k: 1\le k\le N  \}.$$
The generalized Laguerre spectral scheme for \refe{4.1s} is to find $u_{N}\in
X^{\bt}_{N}$, such that $u_N(0)=\eta$ and
\begin{equation}\label{4.3s}
A_{\gamma}(u_{N},v_{N}) =(f,v_{N}),\quad
\forall v_{N}\in X^{0,\bt}_{N}.
\end{equation}

To propose a fully diagonal approximation scheme for \eqref{4.2s} , we need to construct   new basis functions $\{\mathcal{S}_{k}^{\bt}\}_{1\leq k\leq N}$
which are mutually orthogonal with respect to the Sobolev  inner product $A_{\gamma}(\cdot,\cdot)$.

\bglm{lm:4.1s} Let $\mathcal{S}_{k}^{\bt}\in X^{0,\bt}_k,\, k\ge 1$  be the Sobolev orthogonal Laguerre functions such that $\mathcal{S}_{k}^{\bt}-l^{-1,\beta}_k\in X^{0,\bt}_{k-1}$ and
\begin{equation}\label{4.15s}
A_{\gamma}(\mathcal{S}_{k}^{\bt},\mathcal{S}_{m}^{\bt})=   \varrho_{k}\delta_{k,m}, \quad k,m\ge 1.
\end{equation} Then we have
\begin{equation}\label{4.16s}\mathcal{S}_{1}^{\bt}(x)=l_1^{-1,\beta}(x),\quad \mathcal{S}_{k}^{\bt}(x)=l_k^{-1,\beta}(x)-d_{k-1}\mathcal{S}_{k-1}^{\bt}(x), \quad k \ge 2, \end{equation}
where $\varrho_{1}=\dfrac{4\gamma+\beta^2}{2\beta}$ and
 \begin{align*}
& d_{k-1}=\frac{\beta}{4\varrho_{k-1}}-\frac{\gamma}{\beta \varrho_{k-1}}, \quad
 \varrho_{k}= -d^2_{k-1} \varrho_{k-1} +  \frac{\beta}{2}+\frac{2\gamma}{\beta}, \qquad k \ge 2.
\end{align*}
\edlm
\begin{proof}
The proof is in the same way as Lemma \ref{lm:5.1}. We neglect the details.
\end{proof}

To deal with the non-homogenous boundary condition,  we need to
supplement $\mathcal{S}_0^{\beta}(x)$   which is  orthogonal to all functions in  $H^{1}_0(\Lambda)$ with respect to $A_{\gamma}(\cdot,\cdot)$.
Suppose
\begin{align*}
\mathcal{S}_0^{\beta}(x) = \sum_{m=0}^{\infty}\hat s_m l_m^{-1,\beta}(x) \quad \text{ with }  \hat s_0 = 1,
\end{align*}
such that $\mathcal{S}_0^{\beta}(0)=1$.
Then by \eqref{2.22s}, \eqref{2.27s} and \eqref{2.26s},
\begin{align*}
0 =&\, A_{\gamma}(\mathcal{S}_0^{\beta},l_k^{-1,\beta})  = a_{1}^{\alpha,\beta}(\mathcal{S}_0^{\beta},l_k^{-1,\beta}) + \big(\gamma-\frac{\beta^2}{4}\big) (\mathcal{S}_0^{\beta},l_k^{-1,\beta})
\\
=&\,  \dsum_{m=0}^{\infty}  \hat s_m a_{1}^{\alpha,\beta}(l_m^{-1,\beta}, l_k^{-1,\beta})
 + \big(\gamma-\frac{\beta^2}{4}\big)  \dsum_{m=0}^{\infty}  \hat s_m  \big(   l_m^{0,\beta}-l_{m-1}^{0,\beta},  l_k^{0,\beta}-l_{k-1}^{0,\beta} \big)
\\
=&\, \beta  \hat s_k +  \big(\gamma-\frac{\beta^2}{4}\big)  \frac{2\hat s_k - \hat s_{k-1}-\hat s_{k+1}}{\beta}
\\
=&\, \big(\gamma+\frac{\beta^2}{4}\big) \frac{2\hat s_k}{\beta} - \big(\gamma-\frac{\beta^2}{4}\big)  \frac{\hat s_{k-1}+\hat s_{k+1}}{\beta}
,\quad  k\ge 1.
\end{align*}

The  characteristic  equation for the above three term recurrence relation reads
\begin{align*}
(\beta^2-4\gamma\big) z^2 + 2\big(4\gamma+\beta^2 \big) z + (\beta^2-4\gamma\big)=0,
\end{align*}
which admits  two distinct real roots  $z_{\pm}=\displaystyle  \frac{2\sqrt{\gamma}\mp\beta }{ 2\sqrt{\gamma}\pm \beta}
$
if and only if $\gamma> 0$.  In this case,
all  $\hat s_m$, $m\ge 0$ can be expressed as
\begin{align}
\label{coefSeq}
   \hat s_m = c_{+}  \frac{(2\sqrt{\gamma} - \beta )^m}{ (2\sqrt{\gamma}+\beta )^m} + c_{-}  \frac{(2\sqrt{\gamma}+\beta )^m}{( 2\sqrt{\gamma}-\beta  )^m},
\end{align}
with the coefficients $c_{\pm}$ to be determined by $\hat s_0=1$ and $\lim_{m\to \infty} \hat s_m=0$.
As a result,
\begin{align*}
   &c_{+}=1, \quad c_{-}=0,
   \quad \text{ and }\quad
    \mathcal{S}_0^{\beta}(x) = \sum_{m=0}^{\infty} \dfrac{(2\sqrt{\gamma} - \beta )^m}{ (2\sqrt{\gamma}+\beta )^m} l_m^{-1,\beta}(x).
   \end{align*}

Also, we  need a function $\mathcal{S}^{\beta}_{0,N} \in X^{\beta}_N$  which satisfies $\mathcal{S}^{\beta}_{0,N}(0)=1$ and  is  orthogonal to all functions in  $X^{0,\beta}_N$ with respect to $A_{\gamma}(\cdot,\cdot)$. Let us write
\begin{align*}
\mathcal{S}_{0,N}^{\beta}(x) = \sum_{m=0}^{N}\hat s_m l_k^{-1,\beta}(x).
\end{align*}
Then $\{\hat s_m\}_{0\le m\le N}$ is determined by \eqref{coefSeq}  for $1\le m\le N$ together with the endpoint values $\hat s_0=1$ and  $\hat s_{N+1}=0$. Solving the system for  $c_{\pm }$£¬we finally have
\begin{align*}
&c_{+} =  \frac{(2\sqrt{\gamma}+\beta )^{2N+2} }{ (2\sqrt{\gamma}+\beta )^{2N+2} - (2\sqrt{\gamma}-\beta )^{2N+2} }  ,
\qquad
c_{-} =  -\frac{(2\sqrt{\gamma}-\beta )^{2N+2} }{ (2\sqrt{\gamma}+\beta )^{2N+2} - (2\sqrt{\gamma}-\beta )^{2N+2} }  ,
\\
&\mathcal{S}_{0,N}^{\beta}(x) = \sum_{m=0}^{N} \frac{(2\sqrt{\gamma}+\beta )^{2N+2-m} (2\sqrt{\gamma}-\beta )^{m}
- (2\sqrt{\gamma}-\beta )^{2N+2-m} (2\sqrt{\gamma}+\beta )^{m}   }{ (2\sqrt{\gamma}+\beta )^{2N+2} - (2\sqrt{\gamma}-\beta )^{2N+2} }  l_k^{-1,\beta}(x).
\end{align*}

\begin{thm} Let $u$ and $u_N$ be the solution  to  \eqref{4.1s}  and  \eqref{4.3s}, respectively. Then
both $u$ and $u_N$ have the explicit representations in $\{\mathcal{S}^{\beta}_k\}$,
\begin{align*}
    & u(x) = \eta \mathcal{S}_0^{\beta}(x) + \sum_{k=1}^{\infty} \hat u_k \mathcal{S}_k^{\beta}(x),  \qquad
    u_N(x) =\eta \mathcal{S}_{0,N}^{\beta}(x) +  \sum_{k=1}^{N} \hat u_k \mathcal{S}_k^{\beta}(x),
    \\
    &\hat u_k
    = \frac{1}{\varrho_k} A_{\gamma}(u, \mathcal{S}_k^{\beta})
    = \frac{(f, \mathcal{S}^{\beta}_m)}{\varrho_m}, \quad k\ge 1.
\end{align*}
\end{thm}

\section{Convergence analysis}\setcounter{equation}{0} \setcounter{lmm}{0} \setcounter{thm}{0}
In this section, we shall derive  the optimal  error estimate for the spectral methods using generalized Laguerre functions.
To this end, we first conduct some numerical analysis on the orthogonal projections.

\subsection{Orthogonal projections}

Let  $r,N\in \mathbb{N}_0$, $\alpha>-r-1$ and $\beta>0$.
Define the orthogonal projection $\pi^{\alpha,\beta}_{r,N}: H^{r}_{w^{\alpha+r}}(\Lambda)\mapsto X^{\beta}_N$ such that
\begin{align}
\label{projection}
a^{\alpha,\beta}_{r}(u-\pi^{\alpha,\beta}_{r,N}u,v)= 0, \quad v\in X^{\beta}_N.
\end{align}
In view of the orthogonality relation \eqref{2.27s}, $\pi^{\alpha,\beta}_{r,N} u$ is a truncated Fourier series
of $u$
in $\{l^{\alpha,\beta}_k\}$,
\begin{align}
\label{truncated}
 \pi^{\alpha,\beta}_{r,N} u(x) = \sum_{k=0}^{N} \hat u_k l^{\alpha,\beta}_k(x),
\quad u(x) = \sum_{k=0}^{\infty} \hat u_k l^{\alpha,\beta}_k(x), \qquad
\hat u_k = \frac{a^{\alpha,\beta}_r(u, l^{\alpha,\beta}_k)}{a^{\alpha,\beta}_r(l^{\alpha,\beta}_k, l^{\alpha,\beta}_k)}.
\end{align}
In return, one obtains that,  for any nonnegative integers $N$ and $s$ with $N,s>-\alpha-1$,
\begin{align*}
a^{\alpha,\beta}_s(u-\pi^{\alpha,\beta}_{r,N}u,v) =0,\quad v\in X^{\beta}_N,
\end{align*}
which states that $\pi^{\alpha,\beta}_{r,N} = \pi^{\alpha,\beta}_{s,N} $ for all
 admissible $s,r$ and $N$. For this reason, we shall  omit the subscript $r$ and simplify write
$\pi^{\alpha,\beta}_{N}=\pi^{\alpha,\beta}_{r,N}$.

Besides,
 \eqref{2.18s} clearly states $\partial_x^{\ell} l^{-n,\beta}_k(0)=0$,  $\ell=0,1,\dots,n-1$ for any $k\ge n \ge 1$.
 Thus for $N\ge n$, $\pi_N^{-n,\beta}u$ preserves the endpoint values of  $u$ up to the $(n-1)$-th order derivative, i.e.,
\begin{align}
\label{endpointVal}
 \partial_x^{\ell}\pi_N^{-n,\beta}u(0) =  \partial_x^{\ell}  u(0), \qquad \ell=0,1,\dots,n-1.
 \end{align}
   In other words,  $u-\pi_N^{-n,\beta}u\in H^n_0(\Lambda)$  for any  $u\in H^n(\Lambda)$ if $N\ge n$.

To measure the error between $u$ and $\pi_N^{\alpha,\beta}u$, we introduce the equivalent norm in $H^r_{w^{\alpha+r}}(\Lambda)$
 for $r\in \mathbb{N}_0$ and $\alpha\in \mathbb{R}$,
\begin{align*}
\|u\|_{r,\alpha,\beta} = \big[a^{\alpha,\beta}_r(u,u)\big]^{1/2} =\Big[ \sum_{\nu=0}^{r} \binom{r}{\nu} \frac{\beta^{2r-2\nu}}{2^{2r-2\nu}}
\big(\partial_x^{\nu} u,  \partial_x^{\nu} u\big)_{w^{\alpha+r}} \Big]^{1/2}.
\end{align*}

\begin{theorem}\label{th:3.1}
Let  $r\in \mathbb{N}_0$, $\alpha>-r-1$ and $\beta>0$.  Then  for any function $u\in H^r_{w^{\alpha+r}}(\Lambda)$  and any nonnegative integers $N,s>-\alpha-1$,
\begin{align}
\label{errorProj}
\|u-\pi_{N}^{\alpha,\beta}u\|_{s,\alpha,\beta}
  \lesssim  (\beta N)^{\frac{s-r}{2}}  \|u-\pi_{N}^{\alpha,\beta}u\|_{r,\alpha,\beta}
 \lesssim  (\beta N)^{\frac{s-r}{2}}\|u\|_{r,\alpha,\beta},
 \quad r\ge s,
 \end{align}
where  the implicit constants $c=c(\alpha,r)$ is independent of $\beta$, $N$,  $s$ and $u$.
\end{theorem}

\begin{proof} By \eqref{truncated} and  the orthogonality \eqref{2.27s},
\begin{align*}
\|u-\pi^{\alpha,\beta}_N&u\|_{s, \alpha,\beta}^2
=   \sum_{k=N+1 }^{\infty}  \beta^{s-\alpha-1} \gamma^{\alpha}_{k,s} \hat u_k^2
=   \beta^{s-r}   \sum_{k=N+1 }^{\infty} \left(\frac{\gamma^{\alpha}_{k,s}}{\gamma^{\alpha}_{k,r}}\right)  \beta^{r-\alpha-1} \gamma^{\alpha}_{k,r} \hat u_k^2.
\end{align*}
By \eqref{ratio}, one reveals that
\begin{align*}
\frac{\gamma^{\alpha}_{k,s}}{\gamma^{\alpha}_{k,r}}
\lesssim \frac{2^{r-s}}{(2k+\alpha+1)^{r-s}}   \lesssim \frac{1}{N^{r-s}}, \qquad k \ge N+1,\,   r\ge s.
\end{align*}
As a result,
\begin{align*}
\|u-&\pi^{\alpha,\beta}_Nu\|_{s, \alpha,\beta}^2
\lesssim   (\beta N)^{s-r}   \sum_{k=N+1 }^{\infty}  \beta^{r-\alpha-1} \gamma^{\alpha}_{k,r} \hat u_k^2
\\
& =  (\beta N)^{s-r}    \|u-\pi^{\alpha,\beta}_Nu\|_{r, \alpha,\beta}^2
\le  (\beta N)^{s-r} \|u\|_{r, \alpha,\beta}^2,
\end{align*}
which leads to \eqref{errorProj}.
\end{proof}
\subsection{Convergence analysis}
We first give the error estimate of the generalized Laguerre spectral method
\eqref{4.3s} for the non-homogenous  Dirichlet boundary value problem \eqref{4.1s}.

\begin{theorem}\label{th:4.1}
Let $u$ and $u_N$ be the solutions to \eqref{4.1s} and \eqref{4.3s}, respectively.
If $u\in  H^{r}_{w^{r-1}}(\Lambda)$ and integer
$r\geq 1$, then for sufficiently large $N$,
\begin{equation}\label{4.4s}\|(u-u_{N})^\prime\|^{2}+\gamma\|u-u_{N}\|^{2}\lesssim
 \Big( 1+\frac{4 \gamma}{\beta^2} \Big)\,  (\beta N)^{1-r} \|u\|^{2}_{r,-1,\beta} ,\end{equation}
and
\begin{equation}\label{4.5s} \|u-u_{N}\|_{w^{-1}}\lesssim
 \Big(1+\frac{\beta^2}{4\gamma} \Big)
\Big(1+\frac{2\sqrt{\gamma}}{\beta}\Big)^2 \,   (\beta N)^{-\frac{r}{2}} \|u\|_{r,-1,\beta}.
\end{equation}
\end{theorem}
\begin{proof} We first prove the inequality \refe{4.4s}.
By \eqref{4.2s}, we get
\begin{align*}
( (\pi_N^{-1,\beta} &u)^\prime,v_{N}^\prime)+\gamma (\pi_N^{-1,\beta}  u,v_{N})
=  (f,v_{N})
+ ( (\pi_N^{-1,\beta} u-u)^\prime,v_{N}^\prime)+\gamma (\pi_N^{-1,\beta}  u-u,v_{N}), \quad v_N\in X^{0,\beta}_N.
\end{align*}
Subtracting the above equation
from \refe{4.3s} yields
\begin{align*}
 ((\pi^{-1,\beta}_{N}&u-u_{N})^\prime,v_{N}^\prime)+\gamma(\pi^{-1,\beta}_{N} u-u_{N},v_{N})
=( (\pi_N^{-1,\beta} u-u)^\prime,v_{N}^\prime)+\gamma (\pi_N^{-1,\beta}  u-u,v_{N}).
\end{align*}
The above with \refe{projection} and the Cauchy-Schwartz
inequality gives that for any real number $q$,
\begin{align*}
&\,\big|  ((\pi^{-1,\beta}_{N}u-u_{N})^\prime,v_{N}^\prime)+\gamma(\pi^{-1,\beta}_{N} u-u_{N},v_{N})\big|
\\
=&\, \big|  (1+q)( (\pi_N^{-1,\beta} u-u)^\prime,v_{N}^\prime)+ \frac{4\gamma+\beta^2q}{4} (\pi_N^{-1,\beta}  u-u,v_{N})\big|
\\
\le &\, |1+q|\, \big\|(\pi_N^{-1,\beta} u-u)^\prime\big\|\, \big\|v_{N}^\prime\big\| + \frac{|4\gamma+\beta^2q|}{4}
\big\|\pi_N^{-1,\beta} u-u\big\|\, \big\|v_{N}\big\|
\\
\le&\,  \Big[ (1+q)^2 \big\|(\pi_N^{-1,\beta} u-u)^\prime\big\|^2 + \frac{(4\gamma+\beta^2q)^2}{16\gamma } \big\|\pi_N^{-1,\beta} u-u\big\|^2 \Big]^{1/2}\,
\big[ \|v_{N}^\prime\|^2+ \gamma \|v_{N} \|^2\big]^{1/2}.
\end{align*}
Taking $v_{N}=u_{N}-\pi_N^{-1,\beta}u\in X^{0,\beta}_N$, we obtain
\begin{align*}
\|(u_{N}-&\pi^{-1,\beta}_{N}u)^\prime\|^{2}+\gamma\|u_{N}-\pi^{-1,\beta}_{N}u\|^{2}
\le (1+q)^2 \big\|(\pi_N^{-1,\beta} u-u)^\prime\big\|^2 + \frac{(4\gamma+\beta^2q)^2}{16\gamma } \big\|\pi_N^{-1,\beta} u-u\big\|^2.
\end{align*}
This, along with the triangle inequality, leads to
\begin{align*}
&\|(u-u_{N})^\prime\|^{2}+\gamma \|u-u_{N}\|^{2}
\\ & \leq
2\|(u-\pi^{-1,\beta}_{N} u)^\prime\|^{2}+2\gamma \|u-\pi^{-1,\beta}_{N}u\|^{2}+2\|(u_{N}-\pi^{-1,\beta}_{N}u)^\prime\|^{2}+2\gamma \|u_{N}-\pi^{-1,\beta}_{N}u\|^{2}
\\
&\leq 
2 (q^2+2q+2) \big\|(\pi_N^{-1,\beta} u-u)^\prime\big\|^2 + \frac{\beta^4q^2+8\beta^2\gamma q+ 32 \gamma^2}{  8\gamma } \big\|\pi_N^{-1,\beta} u-u\big\|^2.
\end{align*}
Now taking $q=-\frac{2\sqrt{2\gamma}}{\beta}$ and using \eqref{errorProj}, we obtain
\begin{align*}
\|(u-u_{N}&)^\prime\|^{2}+\gamma \|u-u_{N}\|^{2} \le
\big[2+2\big(1-\tfrac{2\sqrt{2\gamma}}{\beta}\big)^2\big]\,
\Big[ \big\|(\pi_N^{-1,\beta} u-u)^\prime\big\|^2 + \frac{\beta^2}{4} \big\|\pi_N^{-1,\beta} u-u\big\|^2\Big].
\\
\lesssim&\, \big(1+ 4\gamma/\beta^2\big)\,  (\beta N)^{1-r} \|u\|^{2}_{r,-1,\beta} .
\end{align*}

We next verify the inequality \refe{4.5s} using a duality
argument. Consider the
auxiliary problem
\begin{equation}\label{4.7s}
-xv^{\prime\prime}(x)+\gamma xv(x)=u(x)-u_{N}(x)~ \text{ in }~ \Lambda, \qquad v(0) = 0. 
\end{equation}
Its weak form is
$$(\varphi^\prime,v^\prime)+\gamma(\varphi,v)=(\varphi,u-u_N)_{w^{-1}},\quad \forall \varphi \in H^{1}_{0}(\Lambda),$$
which admits a unique solution $v\in H_0^1(\Lambda)$.
Moreover, \eqref{4.7s} yields
\begin{align}
\label{stab} \|u-u_{N}\|_{w^{-1}}^2 = (-xv^{\prime\prime}+\gamma xv,
-v^{\prime\prime}+\gamma v ) = \|v^{\prime\prime}\|_{w}^2 + \gamma^2
\|v\|_{w}^2 + 2\gamma
\|v^{\prime}\|_{w}^2=\|v\|^2_{2,-1,2\sqrt{\gamma}},
\end{align}
where the second equality sign is derived by integration by parts.
Further, a direct computation leads to
\begin{align}
\label{tran}\|v\|_{2,-1,\beta}\leq\Big(1+\frac{\beta^2}{4\gamma} \Big) \|v\|_{2,-1,2\sqrt{\gamma}}.\end{align}
Hence, taking $\varphi=u-u_{N}$ and using the Cauchy-Schwartz inequality, we have
\begin{align}\label{errs}
\|u-u_{N} \|_{w^{-1}}^2&=(u-u_{N},u-u_{N})_{w^{-1}} =
((u-u_{N})^\prime,v^\prime)+\gamma (u-u_{N},v)\nonumber
\\
&\overset{\eqref{4.3s}}=((u-u_{N})^\prime,(v-\pi^{-1,\beta}_Nv)^\prime)+\gamma
(u-u_{N},v-\pi^{-1,\beta}_Nv)\nonumber
\\
&\le \Big[ \|(u-u_{N})^\prime\|^2  + \frac{4\gamma^2}{\beta^2}\|
u-u_{N}\|^2\Big]^{\frac12} \,  \Big[
\|(v-\pi^{-1,\beta}_Nv)^\prime\|^2 + \frac{\beta^2}{4}
\|v-\pi^{-1,\beta}_Nv\|^2 \Big]^{\frac12}\nonumber
\\
& \overset{\eqref{errorProj}}\lesssim   (\beta N)^{-1/2}  \|v\|_{2,-1,\beta} \Big[ \|(u-u_{N})^\prime\|^2  + \frac{4\gamma^2}{\beta^2}\| u-u_{N}\|^2\Big]^{\frac12}
\\
&\overset{\eqref{tran}} \lesssim  \Big(1+\frac{\beta^2}{4\gamma} \Big)  (\beta N)^{-1/2}  \|v\|_{2,-1,2\sqrt{\gamma}}\times \Big(1+\frac{2\sqrt{\gamma}}{\beta}\Big)  \Big[ \|(u-u_{N})^\prime\|^2  + \gamma\| u-u_{N}\|^2\Big]^{\frac12}
\nonumber\\
&\overset{\eqref{stab}}\lesssim  \Big(1+\frac{\beta^2}{4\gamma}
\Big)
 (\beta N)^{-1/2} \|u-u_{N}\|_{w^{-1}} \times  \Big(1+\frac{2\sqrt{\gamma}}{\beta}\Big)^2  (\beta N)^{1/2-r/2}   \|u\|_{r,-1,\beta}
\nonumber\\
& \overset{\eqref{4.4s}} \lesssim   \Big(1+\frac{\beta^2}{4\gamma}
\Big) \Big(1+\frac{2\sqrt{\gamma}}{\beta}\Big)^2    (\beta
N)^{-\frac{r}{2}}  \|u-u_{N}\|_{w^{-1}}
\|u\|_{r,-1,\beta},\nonumber
\end{align}
which ends the proof of \eqref{4.5s}.
\end{proof}

Next, we present the main theorem on the generalized Laguerre spectral method for the Robin boundary value problem
\eqref{5.1}.
\begin{theorem}\label{th:4.2}
Let $u$ and $u_N$ be the solutions to \eqref{5.1} and \eqref{5.3}, respectively.
If $u\in  H^{r}_{w^{r-1}}(\Lambda)$ and integer
$r\geq 1$, then for sufficiently large $N$,
\begin{equation}\label{4.4sr}
\mu \big|u(0)-u_N(0)\big|^2+\|(u-u_{N})^\prime\|^{2}+\gamma\|u-u_{N}\|^{2}\lesssim
 \Big( 1+\frac{4 \gamma}{\beta^2} \Big)\,   (\beta N)^{1-r} \|u\|^{2}_{r,-1,\beta},\end{equation}
and
\begin{equation}\label{4.5sr} \|u-u_N\|_{(1+w)^{-1}}\lesssim
 \Big(1+\frac{1}{\sqrt{2\mu}}\Big)\Big(1+\frac{\beta^2}{4\gamma} \Big)
\Big(1+\frac{2\sqrt{\gamma}}{\beta}\Big)^2 \,   (\beta N)^{-\frac{r}{2}} \|u\|_{r,-1,\beta}.
\end{equation}
\end{theorem}
\begin{proof} By \eqref{projection} and \eqref{5.2},
\begin{align*}
\mu &\pi_N^{-1,\beta}u(0) v_N(0) +    ((\pi_N^{-1,\beta}u)^{\prime}, v_N^{\prime}) + \gamma(  \pi_N^{-1,\beta}u , v_N)
\\
&= \mu u(0) v_N(0)  +  (u^{\prime}, v_N^{\prime}) +  \gamma ( u, v_N) + ( (\pi_N^{-1,\beta} u-u)^\prime,v_{N}^\prime)+\gamma (\pi_N^{-1,\beta}  u-u,v_{N})
\\
&= (f,v_N) + \eta v_N(0) +( (\pi_N^{-1,\beta} u-u)^\prime,v_{N}^\prime)+\gamma (\pi_N^{-1,\beta}  u-u,v_{N}), \quad v_N\in X_N^{\beta}.
\end{align*}
Substracting the above equation by \eqref{5.3} yields
\begin{align*}
\mu (\pi_N^{-1,\beta}&u(0)-u_N(0)) v_N(0) +    ((\pi_N^{-1,\beta}u-u_N)^{\prime}, v_N^{\prime}) + \gamma(  \pi_N^{-1,\beta}u-u_N , v_N)
\\
 =&\,
  ( (\pi_N^{-1,\beta} u-u)^\prime,v_{N}^\prime)+\gamma (\pi_N^{-1,\beta}  u-u,v_{N}).
\end{align*}
Then a similar argument as in the proof  of  \eqref{4.4s}  gives \eqref{4.4sr}.

We next verify the inequality \refe{4.5sr} using a duality argument.
For any $g\in L^2_{w^{-1}+w^{-2}}(\Lambda)$, we consider the auxiliary
problem
\begin{equation}\label{4.7sr}
-xv^{\prime\prime}(x)+\gamma xv(x)=g(x)~ \text{ in }~ \Lambda, \qquad -v^{\prime}(0)+\mu v(0) = 0.
\end{equation}
Its weak form is
$$\mu\varphi(0)v(0)+(\varphi^\prime,v^\prime)+\gamma(\varphi,v)=(\varphi,g)_{w^{-1}},\quad \forall \varphi \in H^{1}(\Lambda),$$
which admits a unique solution $v\in H^1(\Lambda)$.
Taking $\varphi=u-u_{N}$, we have
\begin{equation}\label{sr1}\mu(u(0)-u_N(0))v(0)+(u^\prime-u_N^\prime,v^\prime)+\gamma(u-u_{N},v)=(u-u_{N},g)_{w^{-1}}.\end{equation}
Moreover, by \refe{5.2} and \refe{5.3} we get
\begin{equation}\label{sr2}\mu(u(0)-u_N(0))v_N(0)+(u^\prime-u_N^\prime,v_N^\prime)+\gamma(u-u_{N},v_N)=0,\quad
\forall v_{N}\in X^{\bt}_{N}.\end{equation} Let
$v_N=\pi_N^{-1,\beta}v$ in \refe{sr2}. Then by \refe{sr1},
\refe{sr2} and \refe{endpointVal} we deduce
\begin{equation}\label{sr3}((u-u_N)^\prime,(v-\pi_N^{-1,\beta}v)^\prime)+\gamma(u-u_{N},v-\pi_N^{-1,\beta}v)=(u-u_{N},g)_{w^{-1}}.\end{equation}
Further by integration by parts, \eqref{4.7sr} yields
\begin{align}
\label{stabr} \|g\|_{w^{-2}}^2 &= (-v^{\prime\prime}+\gamma v,
-v^{\prime\prime}+\gamma v ) = \|v^{\prime\prime}\|^2 + \gamma^2
\|v\|^2 + 2\gamma \|v^{\prime}\|^2+2\gamma\mu
v^2(0),
\end{align}
and
\begin{align}
\label{stabr1} \|g\|_{w^{-1}}^2 &= (-xv^{\prime\prime}+\gamma xv,
-v^{\prime\prime}+\gamma v ) = \|v^{\prime\prime}\|^2_{w} + \gamma^2
\|v\|^2_{w} + 2\gamma \|v^{\prime}\|^2_{w}-\gamma v^2(0).
\end{align}
%
Hence
\begin{align}
\label{stabr2} \|v\|^2_{2,-1,2\sqrt{\gamma}}=\|v^{\prime\prime}\|^2_{w} + \gamma^2
\|v\|^2_{w} + 2\gamma \|v^{\prime}\|^2_{w} \leq \|g\|_{w^{-1}}^2+\frac{1}{2\mu}\|g\|_{w^{-2}}^2\leq (1+\frac{1}{2\mu})\|g\|_{w^{-1}+w^{-2}}^2.
\end{align}
Thus, a similar argument as \eqref{errs}  gives
\begin{align*}
|(u-u_{N},g)_{w^{-1}}|&\overset{\eqref{sr3}}=|((u-u_{N})^\prime,(v-\pi^{-1,\beta}_Nv)^\prime)+\gamma
(u-u_{N},v-\pi^{-1,\beta}_Nv)|
\\
&\overset{\eqref{errs}} \lesssim  \Big(1+\frac{\beta^2}{4\gamma}
\Big)  (\beta N)^{-1/2}  \|v\|_{2,-1,2\sqrt{\gamma}}\times
\Big(1+\frac{2\sqrt{\gamma}}{\beta}\Big)  \Big[
\|(u-u_{N})^\prime\|^2  + \gamma\| u-u_{N}\|^2\Big]^{\frac12}
\\
&\overset{\eqref{stabr2}}\lesssim  \Big(1+\frac{1}{\sqrt{2\mu}}\Big)\Big(1+\frac{\beta^2}{4\gamma}
\Big)
 (\beta N)^{-1/2}  \|g\|_{w^{-1}+w^{-2}} \times  \Big(1+\frac{2\sqrt{\gamma}}{\beta}\Big)^2  (\beta N)^{1/2-r/2}   \|u\|_{r,-1,\beta}
\\
&  \lesssim   \Big(1+\frac{1}{\sqrt{2\mu}}\Big)\Big(1+\frac{\beta^2}{4\gamma} \Big)
\Big(1+\frac{2\sqrt{\gamma}}{\beta}\Big)^2    (\beta
N)^{-\frac{r}{2}}  \|g\|_{w^{-1}+w^{-2}}  \|u\|_{r,-1,\beta}.
\end{align*}
Finally, we obtain
\begin{align*}&\|u-u_N\|_{(1+w)^{-1}}=\|\frac{x}{1+x}(u-u_N)\|_{w^{-1}+w^{-2}}=\sup_{g\in L^2_{w^{-1}+w^{-2}}(\Lambda),~g\neq 0}\frac{|(\frac{x}{1+x}(u-u_N),g)_{w^{-1}+w^{-2}}|}{\|g\|_{w^{-1}+w^{-2}}}
\\&=\sup_{g\in L^2_{w^{-1}+w^{-2}}(\Lambda),~g\neq 0}\frac{|(u-u_N,g)_{w^{-1}}|}{\|g\|_{w^{-1}+w^{-2}}}
\lesssim  \Big(1+\frac{1}{\sqrt{2\mu}}\Big) \Big(1+\frac{\beta^2}{4\gamma} \Big)
\Big(1+\frac{2\sqrt{\gamma}}{\beta}\Big)^2    (\beta
N)^{-\frac{r}{2}} \|u\|_{r,-1,\beta},\end{align*}
 which ends the proof of \eqref{4.5sr}.
\end{proof}

\section{Numerical experiments}\setcounter{equation}{0} \setcounter{lmm}{0} \setcounter{thm}{0} \setcounter{table}{0}
In this section, we examine the effectiveness and the accuracy of  the fully diagonalized Laguerre spectral method
for solving second order elliptic equations on the half line.
The righthand terms $\{(f,\mathcal{R}_k^{\beta})\}_{k=0}^N$ or  $\{(f,\mathcal{S}_k^{\beta})\}_{k=1}^N,$  as well as the discrete errors,
are evaluated through the Laguerre-Gauss quadrature with $2N+1$ nodes (cf. \cite{WGW}).

\subsection{Dirichlet boundary value  problems}
We first  examine  the fully diagonalized Laguerre spectral method for the Dirichlet boundary value problem
 \refe{4.1s}.   We take $\gamma=1$, $\eta=1$ in \refe{4.1s} and consider the following  three cases
of the smooth solutions with different decay properties.
 \begin{itemize}
\item $u(x)=e^{-x}(\sin(2x)+\eta),$ which is exponential decay with oscillation.
In Figures \ref{fig1}, \ref{fig2} and \ref{fig1.1}, we plot the
$\log_{10}$ of the discrete $L^2$-, $H^1$- and $L^2_{w^{-1}}$-errors
vs. $N$ with various values of $\beta$. Clearly, the approximate
solutions converge at exponential rates. We also see that for fixed
$N$, the scheme with $\beta=4$ produces better numerical results
than that with $\beta=1,~2.$ However, the choice of the optimal
scaling factor $\beta$ for a given differential equation is still an
open problem. Generally speaking, the choice of $\beta$ depends on
the asymptotic behavior of solutions.

\item $u(x)=(x+\eta)(1+x)^{-h}$ with $h>1,$ which is
algebraic decay. In Figures \ref{fig3}, \ref{fig4} and \ref{fig2.1},
we plot the $\log_{10}$ of the discrete $L^2$-, $H^1$- and
$L^2_{w^{-1}}$-errors vs. $\log_{10} (\beta N)$ with fixed $\bt=4$
and various values of $h$. They show that the faster the exact
solution decays, the smaller the numerical errors would be. Clearly,
\begin{align*}
 \|(x+\eta)(1+x)^{-h}\|_{r,-1,\beta}<\infty,\quad \forall 0\le  r<2h-2.
\end{align*}
According to \reft{th:4.1}, the expected $L^2$- and $H^1$-errors can
be bounded by $c(\beta N)^{3/2-h+\varepsilon}$ for any
$\varepsilon>0$, and the expected $L^2_{w^{-1}}$-error can be
bounded by $c(\beta N)^{1-h+\varepsilon}$ for any $\varepsilon>0$.
The observed convergence rates plotted in Figures \ref{fig3},
\ref{fig4} and \ref{fig2.1} agree well with the theoretical results.
\item $u(x)=(\sin(2x)+\eta)(1+x)^{-h}$, which is
algebraic decay with oscillation. In Figures \ref{fig5}, \ref{fig6}
and \ref{fig5.1}, we plot the $\log_{10}$ of the discrete $L^2$-,
$H^1$- and $L^2_{w^{-1}}$-errors vs. $\log_{10} (\beta N)$ with
fixed $\bt=4$ and various values of $h$.
Since
\begin{align*}
 \|(\sin(2x)+\eta)(1+x)^{-h}\|_{r,-1,\beta}<\infty,\quad \forall 0\le  r<2h,
\end{align*}
the expected $L^2$-, $H^1$- (resp. $L^2_{w^{-1}}$-) errors  given by
\reft{th:4.1} can be bounded by $c(\beta N)^{1/2-h+\varepsilon}$
(resp. $c(\beta N)^{-h+\varepsilon}$) for any $\varepsilon>0$. The
observed convergence rates plotted in Figures \ref{fig5}, \ref{fig6}
and \ref{fig5.1} also agree well with the theoretical results.
\end{itemize}

\begin{figure}[htbp]
\begin{minipage}{0.45\linewidth}
\centering
\includegraphics [height=2.0in]{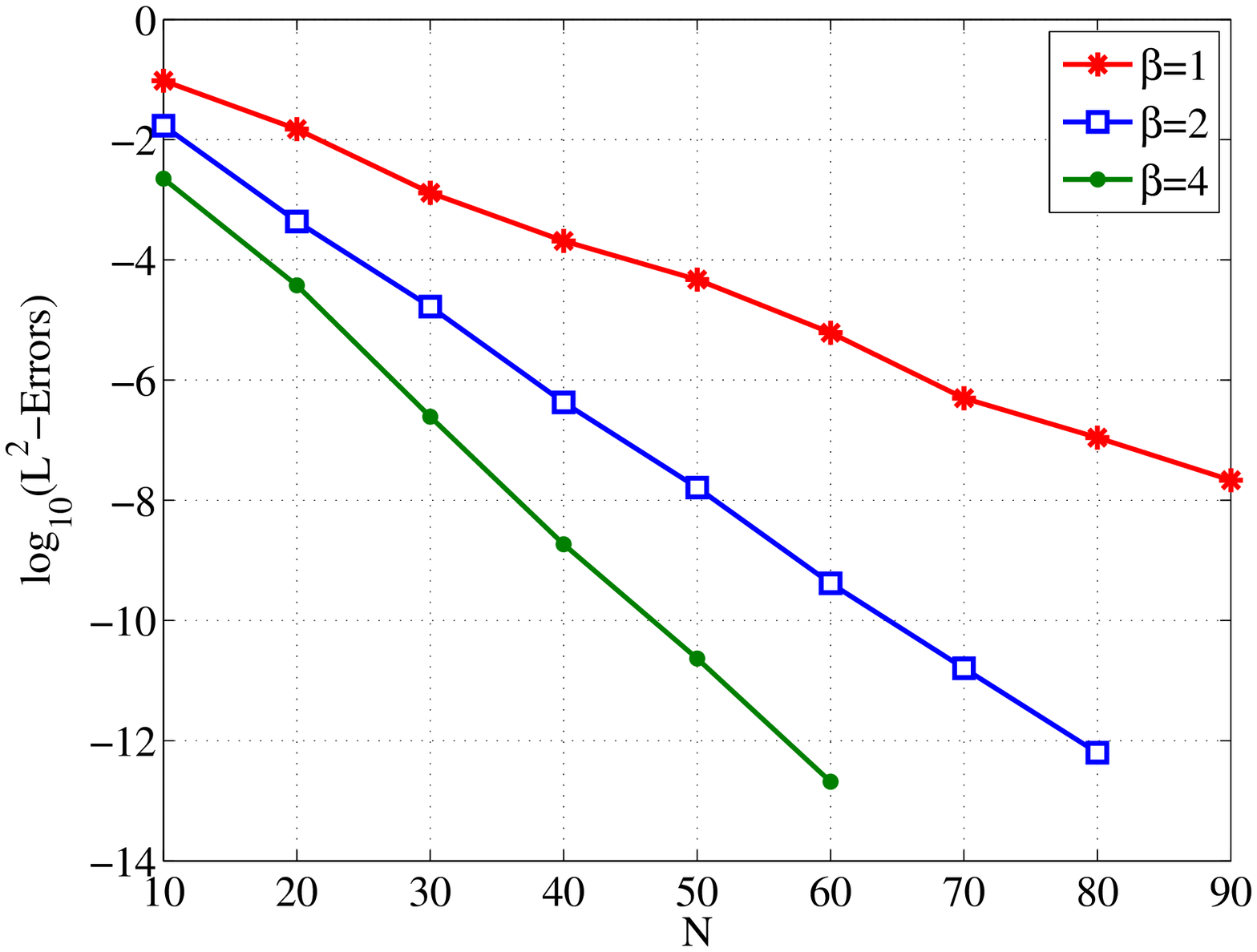}
\caption{\footnotesize $L^2$-errors.\label{fig1}}
\end{minipage}\qquad
\begin{minipage}{0.45\linewidth}
\centering
\includegraphics [height=2.0in]{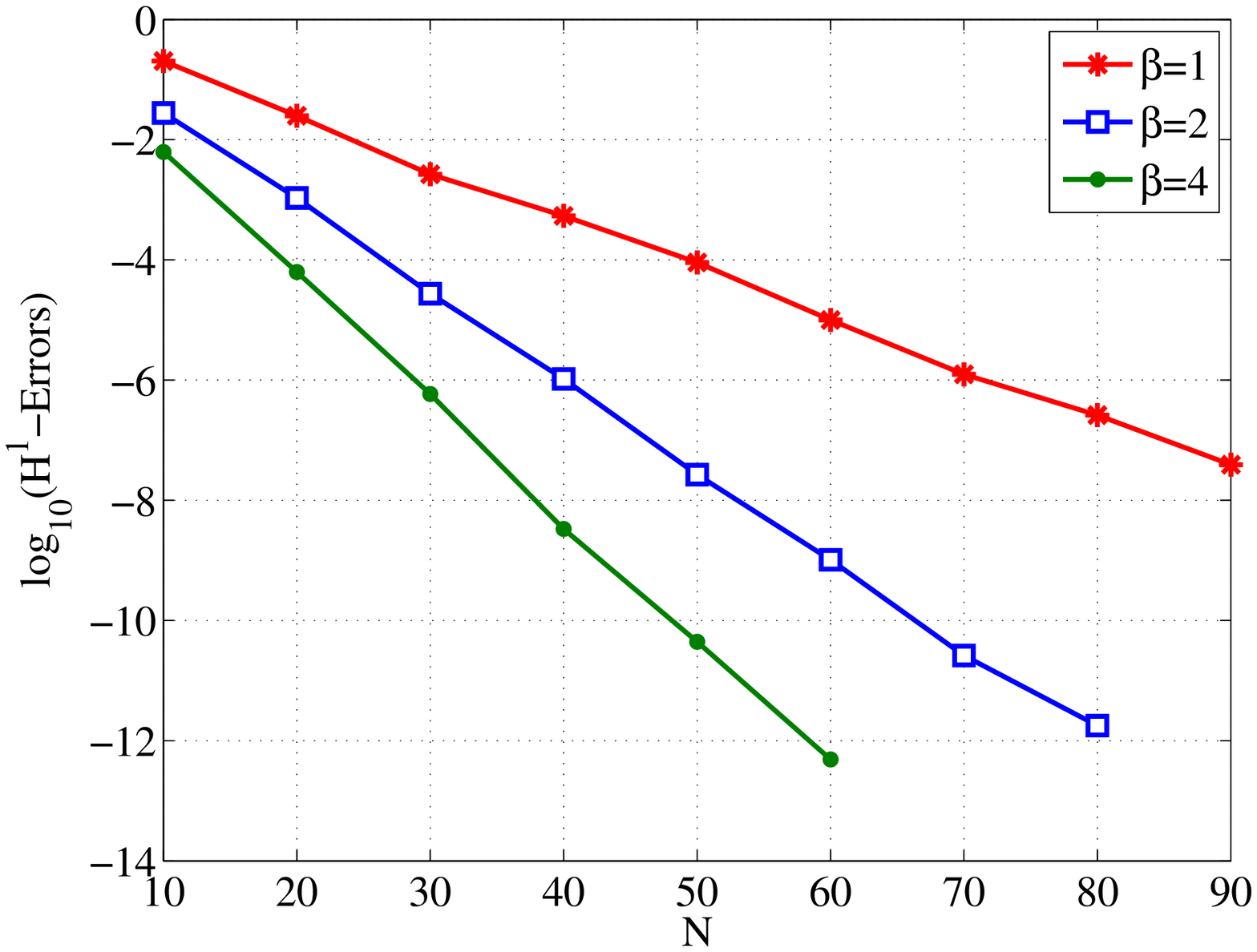}
\caption{\footnotesize $H^1$-errors.\label{fig2}}
\end{minipage}
\end{figure}

\begin{figure}[htbp]
\begin{minipage}{0.45\linewidth}
\centering
\includegraphics [height=2.0in]{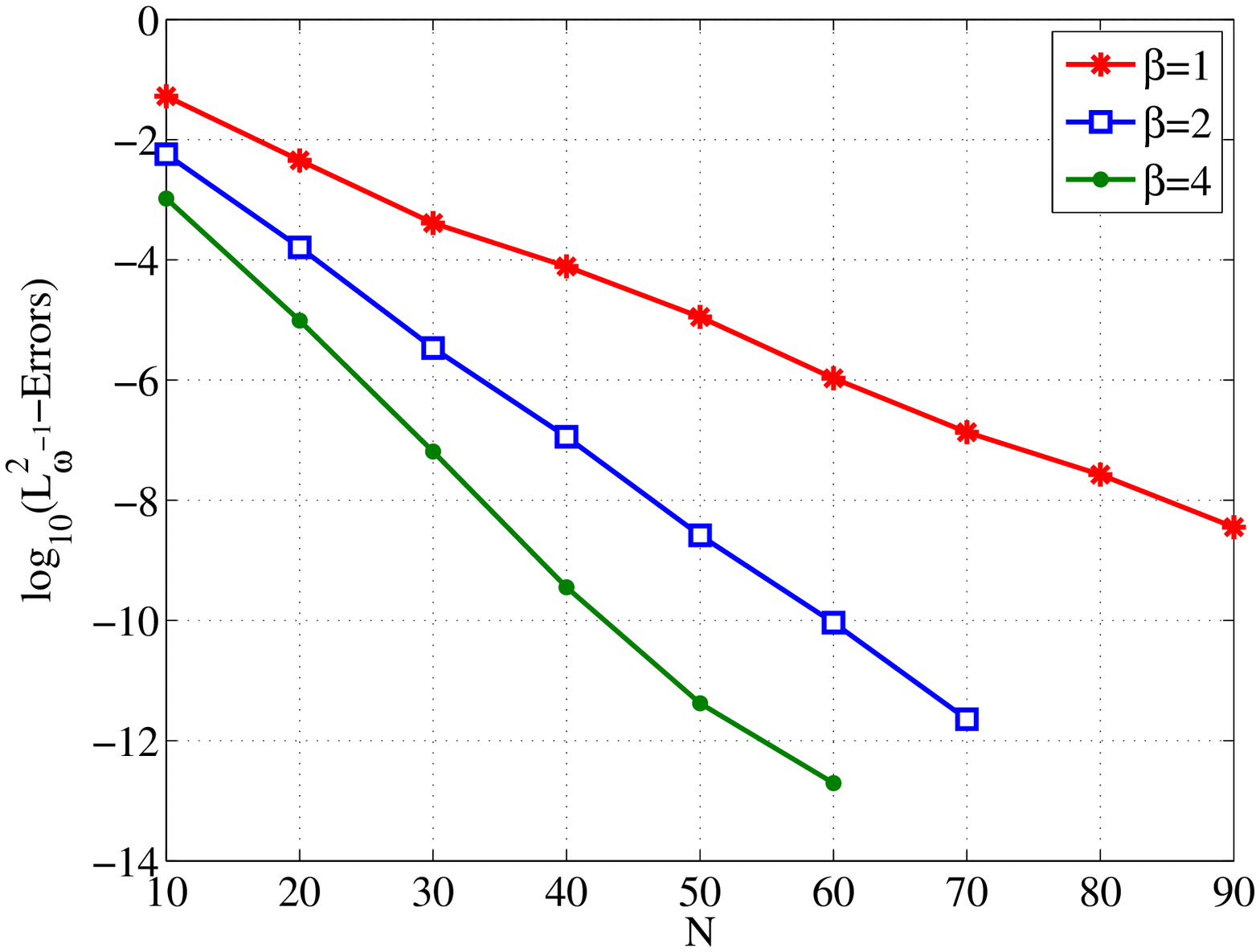}
\caption{\footnotesize $L^2_{w^{-1}}$-errors.\label{fig1.1}}
\end{minipage}\qquad
\begin{minipage}{0.45\linewidth}
\centering
\includegraphics [height=2.1in]{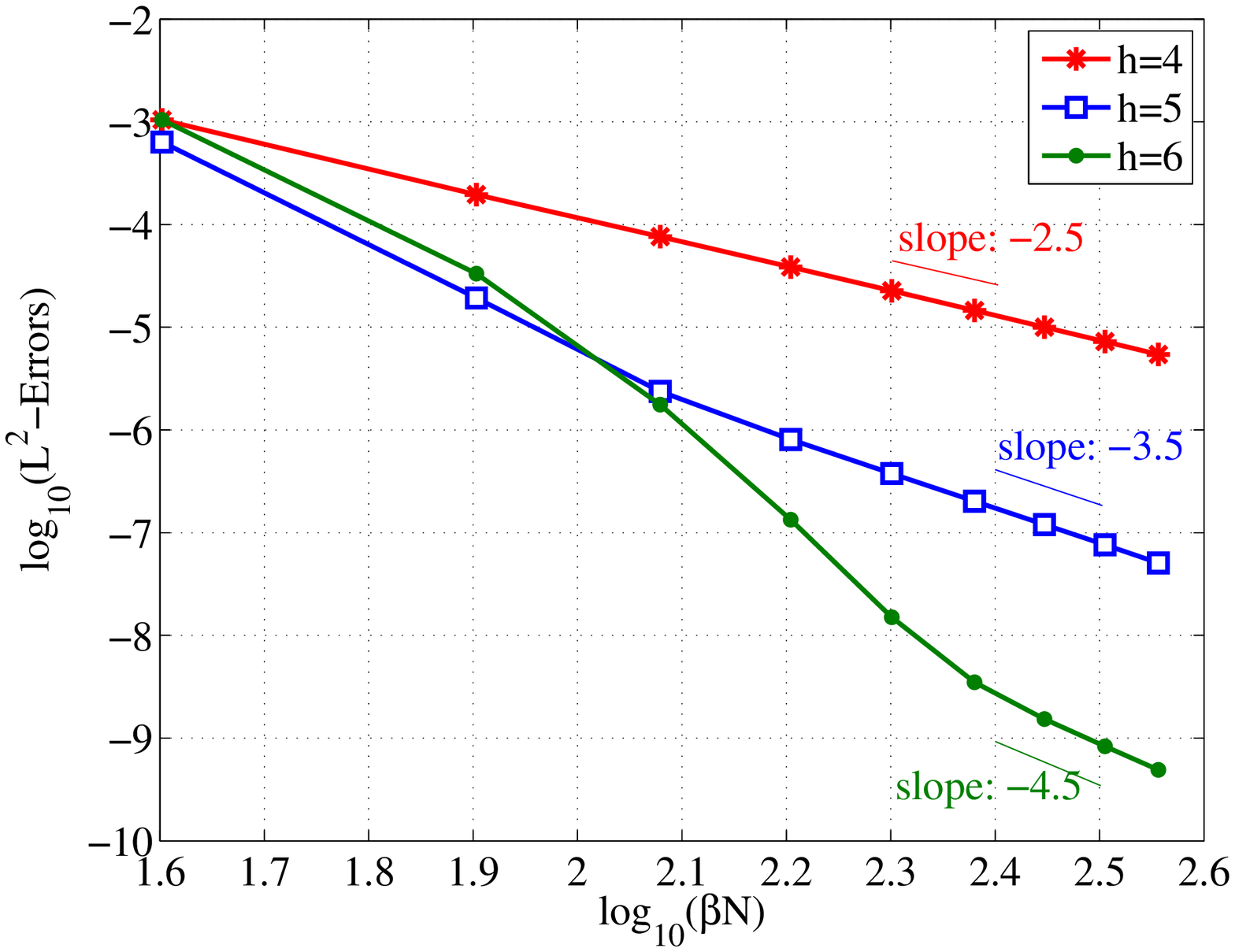}
\caption{\footnotesize $L^2$-errors with $\beta=4$.\label{fig3}}
\end{minipage}
\end{figure}

\begin{figure}[htbp]
\begin{minipage}[t]{0.45\linewidth}
\centering
\includegraphics [height=2.1in]{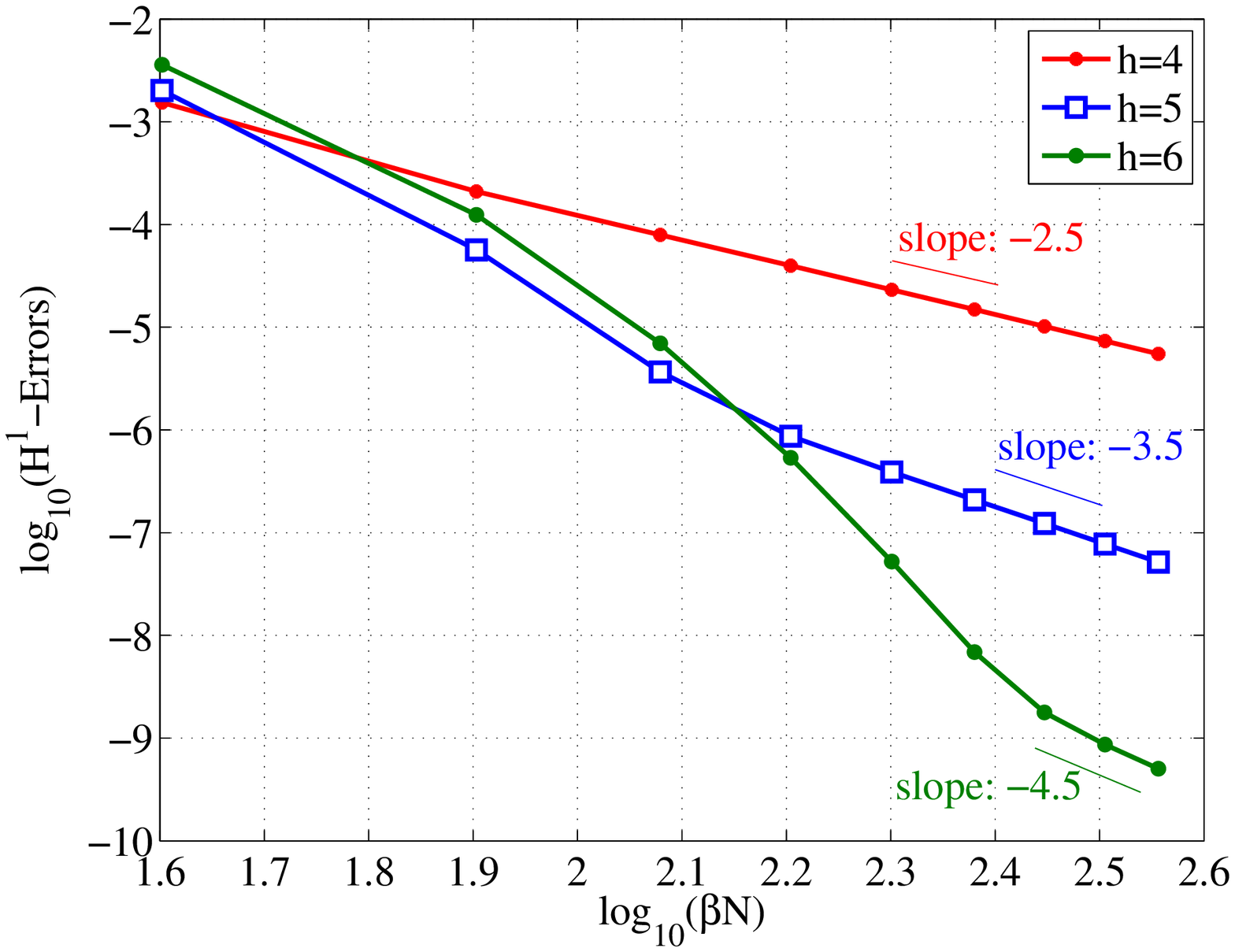}
\caption{\footnotesize $H^1$-errors with $\beta=4$.\label{fig4}}
\end{minipage}\qquad
\begin{minipage}[t]{0.45\linewidth}
\centering
\includegraphics [height=2.1in]{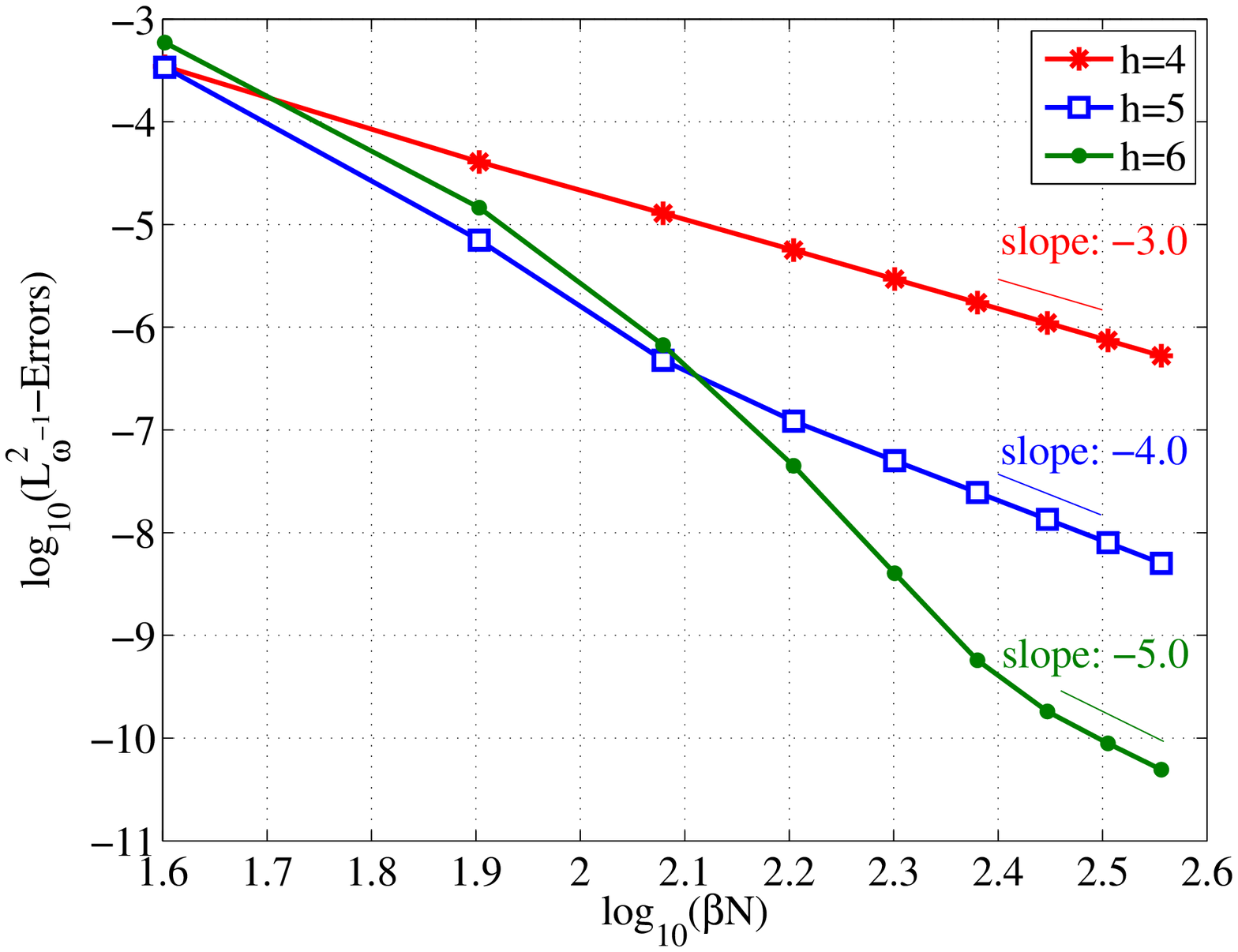}
\caption{\footnotesize $L^2_{w^{-1}}$-errors with
$\beta=4$.\label{fig2.1}}
\end{minipage}
\end{figure}

\begin{figure}[htbp]
\begin{minipage}[t]{0.45\linewidth}
\centering
\includegraphics [height=2.10in]{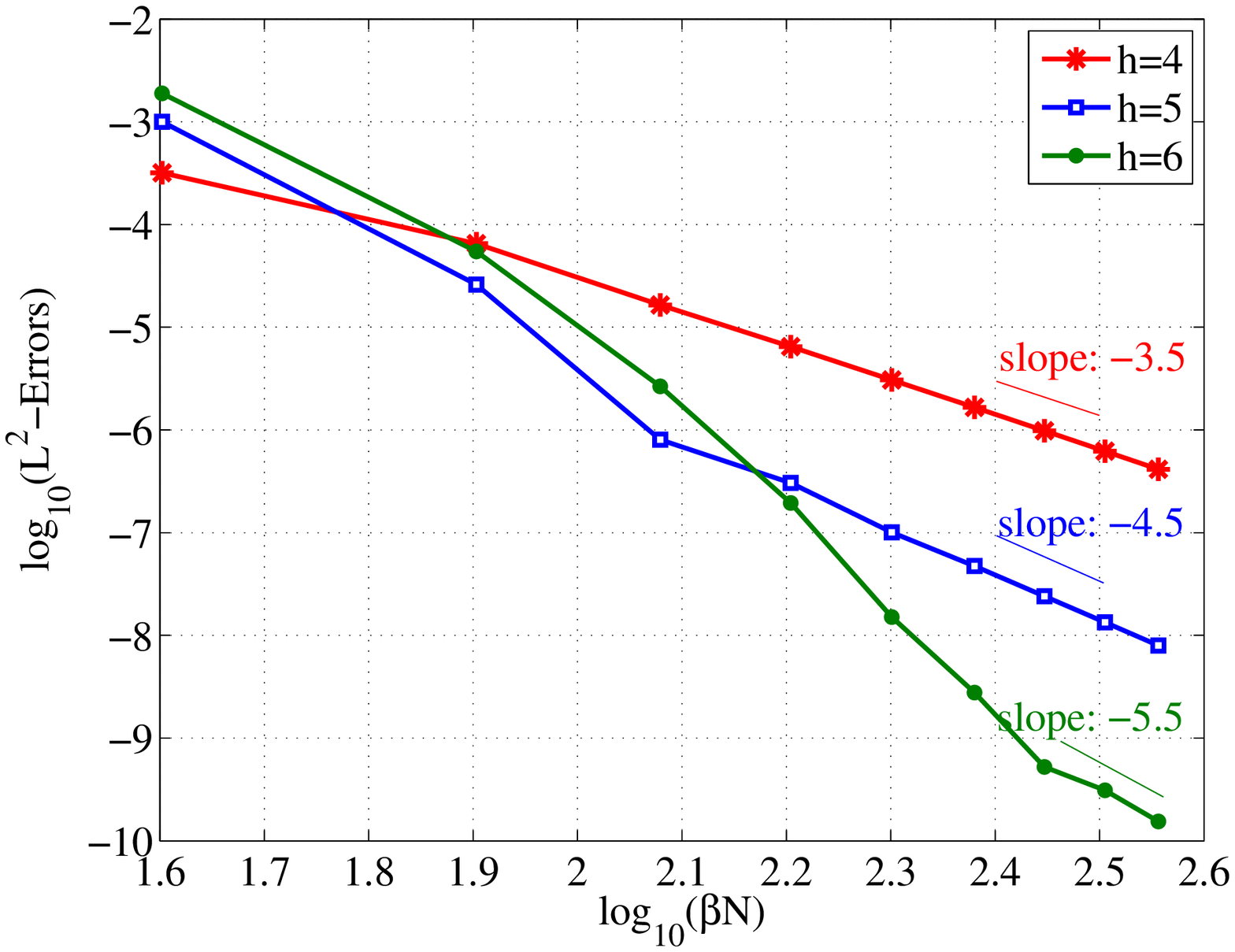}
\caption{\footnotesize $L^2$-errors with $\beta=4$.\label{fig5}}
\end{minipage}\qquad
\begin{minipage}[t]{0.45\linewidth}
\centering
\includegraphics [height=2.10in]{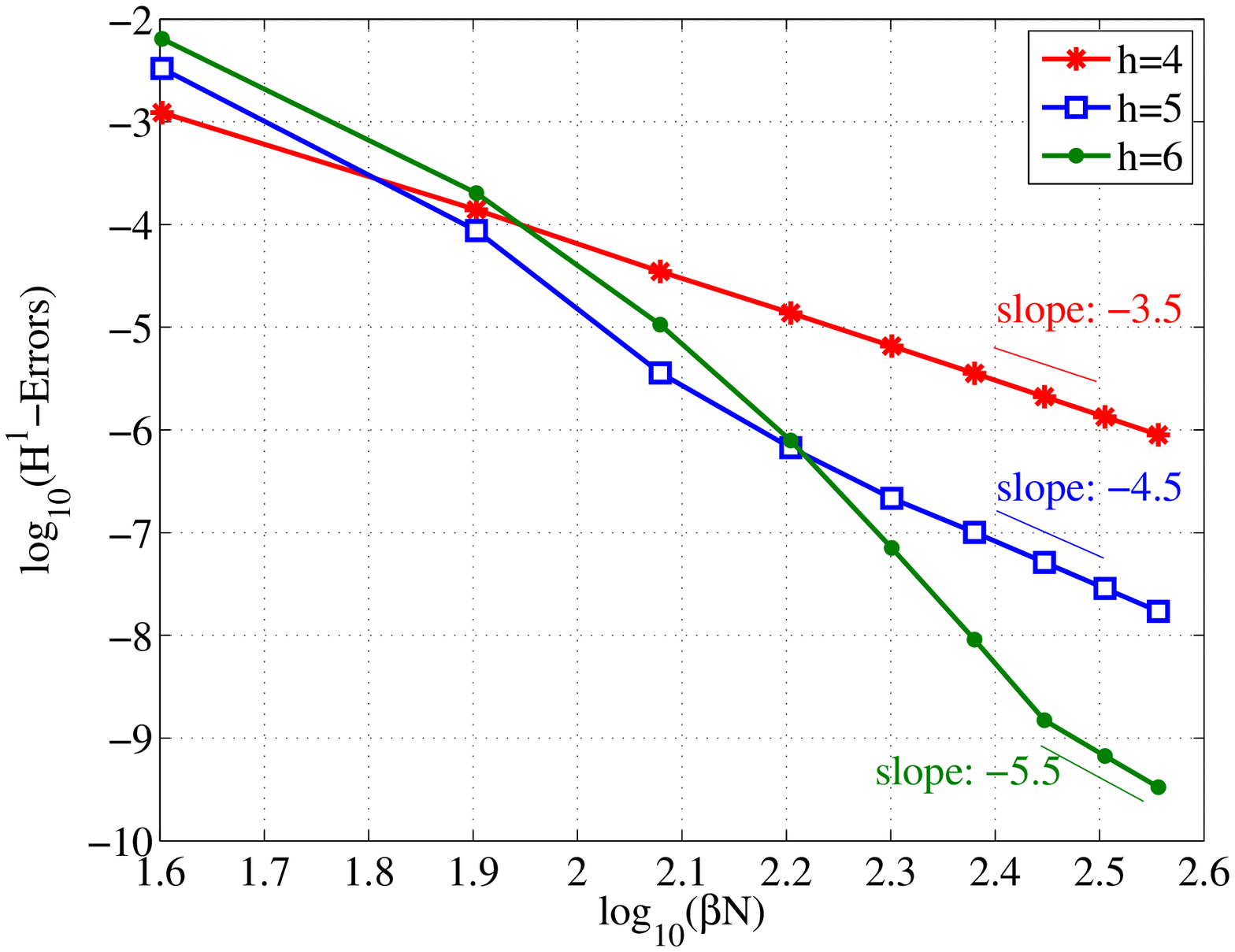}
\caption{\footnotesize $H^1$-errors with $\beta=4$.\label{fig6}}
\end{minipage}
\end{figure}

\begin{figure}[htbp]
\begin{minipage}[t]{0.45\linewidth}
\centering
\includegraphics [height=2.10in]{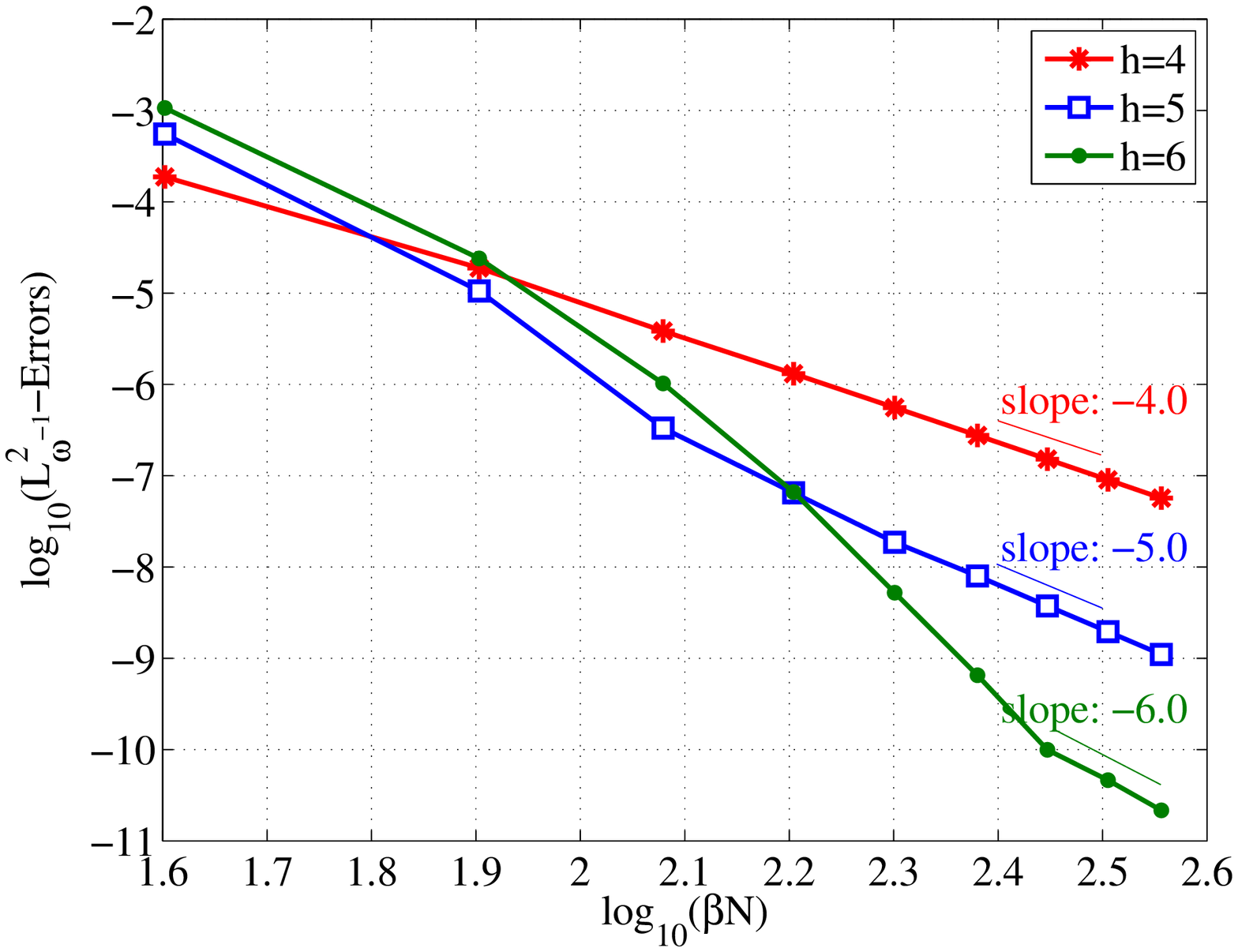}
\caption{\footnotesize $L^2_{w^{-1}}$-errors with
$\beta=4$.\label{fig5.1}}
\end{minipage}\qquad
\begin{minipage}[t]{0.45\linewidth}
\centering
\includegraphics [height=2.0in]{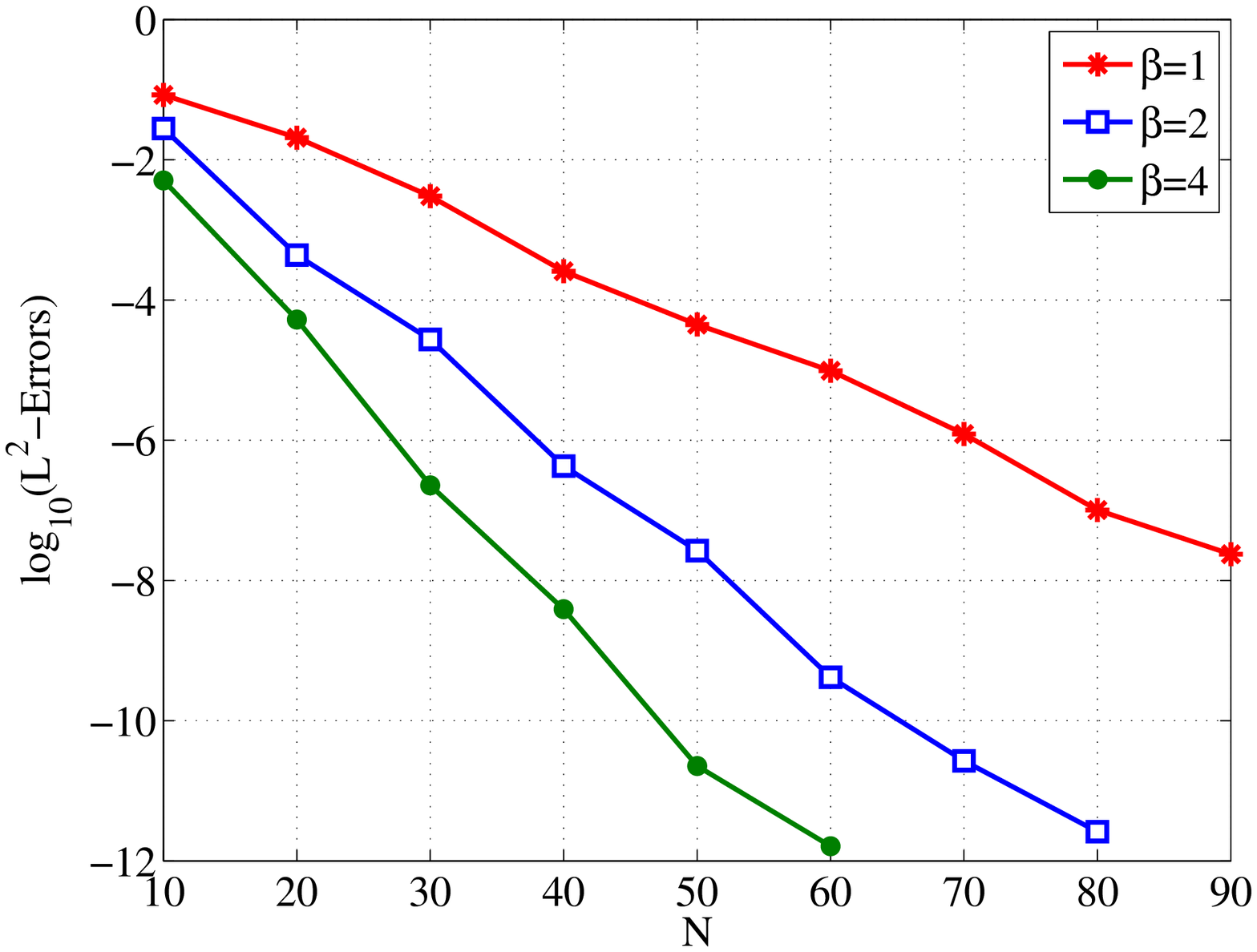}
\caption{\footnotesize $L^2$-errors.\label{fig7}}
\end{minipage}
\end{figure}

\subsection{Robin boundary value problems}
We take $\mu=1$, $\gamma=1$ and $\eta=0$  for  the Robin boundary value problem \eqref{5.1} and  consider two examples with different decay properties
for the exact solutions.

 \begin{itemize}
\item
$u(x)=e^{-x}(\sin(2x)+\cos(2x)),$ which is exponential decay with
oscillation. In Figures \ref{fig7}, \ref{fig8} and \ref{fig8.1}, we
plot the $\log_{10}$ of the discrete $L^2$-, $H^1$- and
$L^2_{(1+w)^{-1}}$-errors vs. $N$ with various values of $\beta$.
Clearly, the approximate solutions converge at exponential rates.
\end{itemize}
\begin{itemize}
\item
$u(x)=(\sin(2x)+\cos(2x))(h+x)^{-h},$ which is algebraic decay with
oscillation. In Figures \ref{fig9}, \ref{fig10} and \ref{fig9.1}, we
plot the $\log_{10}$ of the discrete $L^2$-, $H^1$- and
$L^2_{(1+w)^{-1}}$-errors vs. $\log_{10}(\beta N)$ with fixed $\bt=4$
and
various values of $h$. 
Since
\begin{align*}
 \|(\sin(2x)+\cos(2x))(h+x)^{-h}\|_{r,-1,\beta}<\infty,\quad \forall 0\le  r<2h,
\end{align*}
the expected $L^2$- and $H^1$- (resp. $L^2_{(1+w)^{-1}}$-) errors given by \reft{th:4.2} can be
bounded by $c(\beta N)^{1/2-h+\varepsilon}$ (resp. $c(\beta N)^{-h+\varepsilon}$) for any $\varepsilon>0$.
Again, the observed convergence rates plotted in Figures \ref{fig9},
\ref{fig10} and \ref{fig9.1} agree well with the theoretical results.
\end{itemize}
\begin{figure}[htbp]
\begin{minipage}[t]{0.45\linewidth}
\centering
\includegraphics [height=1.95in]{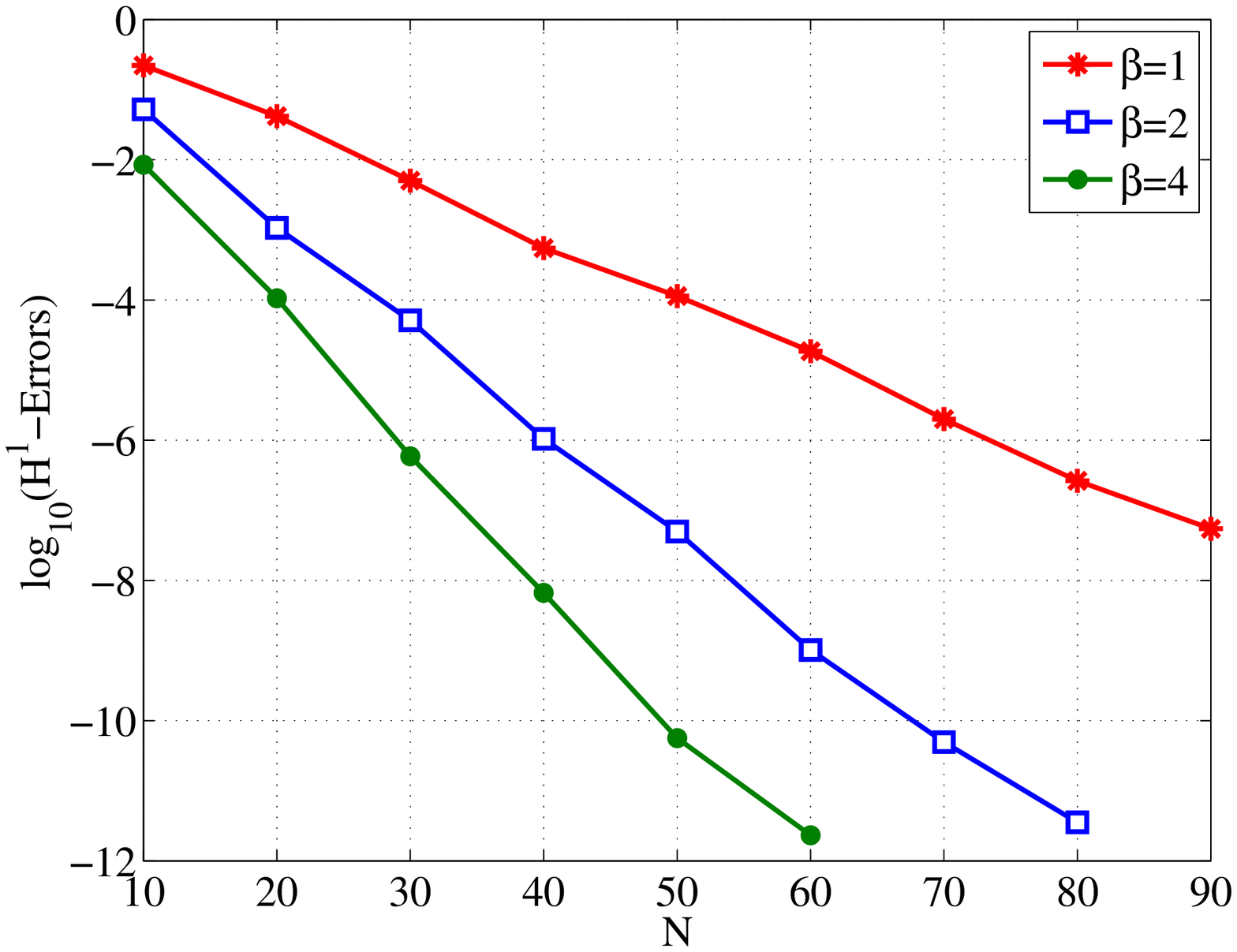}
\caption{\footnotesize $H^1$-errors.\label{fig8}}
\end{minipage}\qquad
\begin{minipage}[t]{0.45\linewidth}
\centering
\includegraphics [height=1.95in]{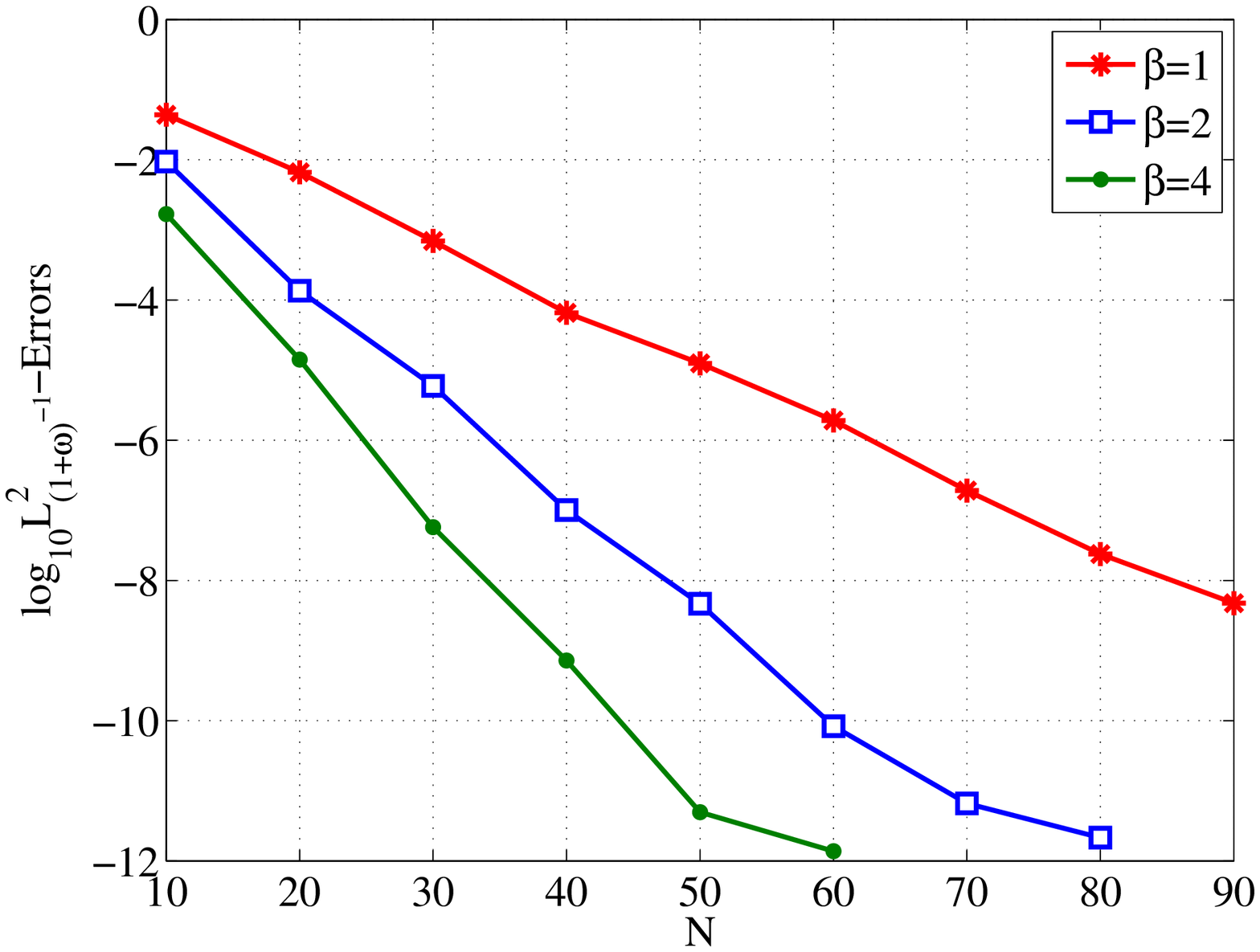}
\caption{\footnotesize $L^2_{(1+w)^{-1}}$-errors.\label{fig8.1}}
\end{minipage}
\end{figure}
\begin{figure}[htbp]
\begin{minipage}[t]{0.45\linewidth}
\centering
\includegraphics [height=2.10in]{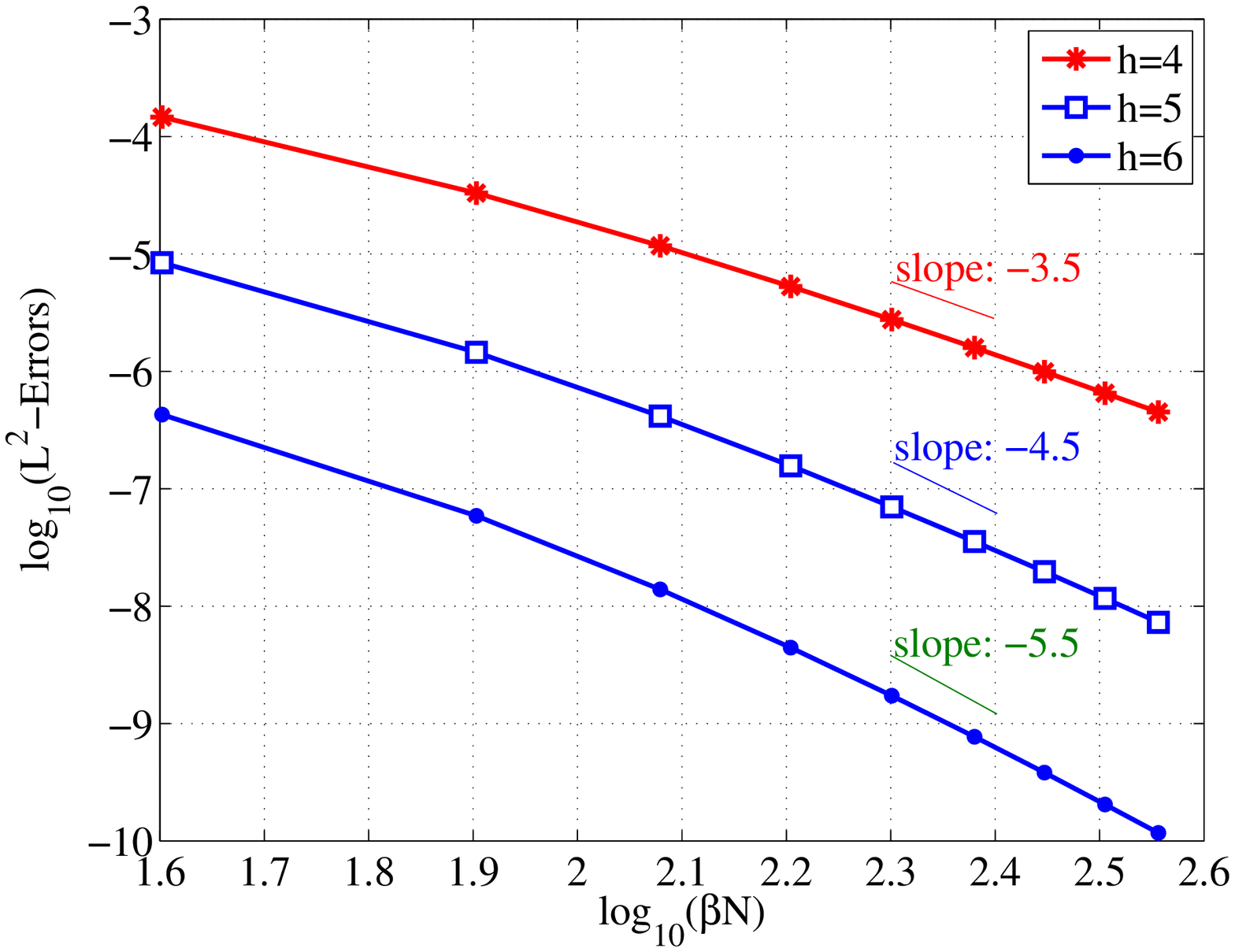}
\caption{\footnotesize $L^2$-errors with $\beta=4$.\label{fig9}}
\end{minipage}\qquad
\begin{minipage}[t]{0.45\linewidth}
\centering
\includegraphics [height=2.10in]{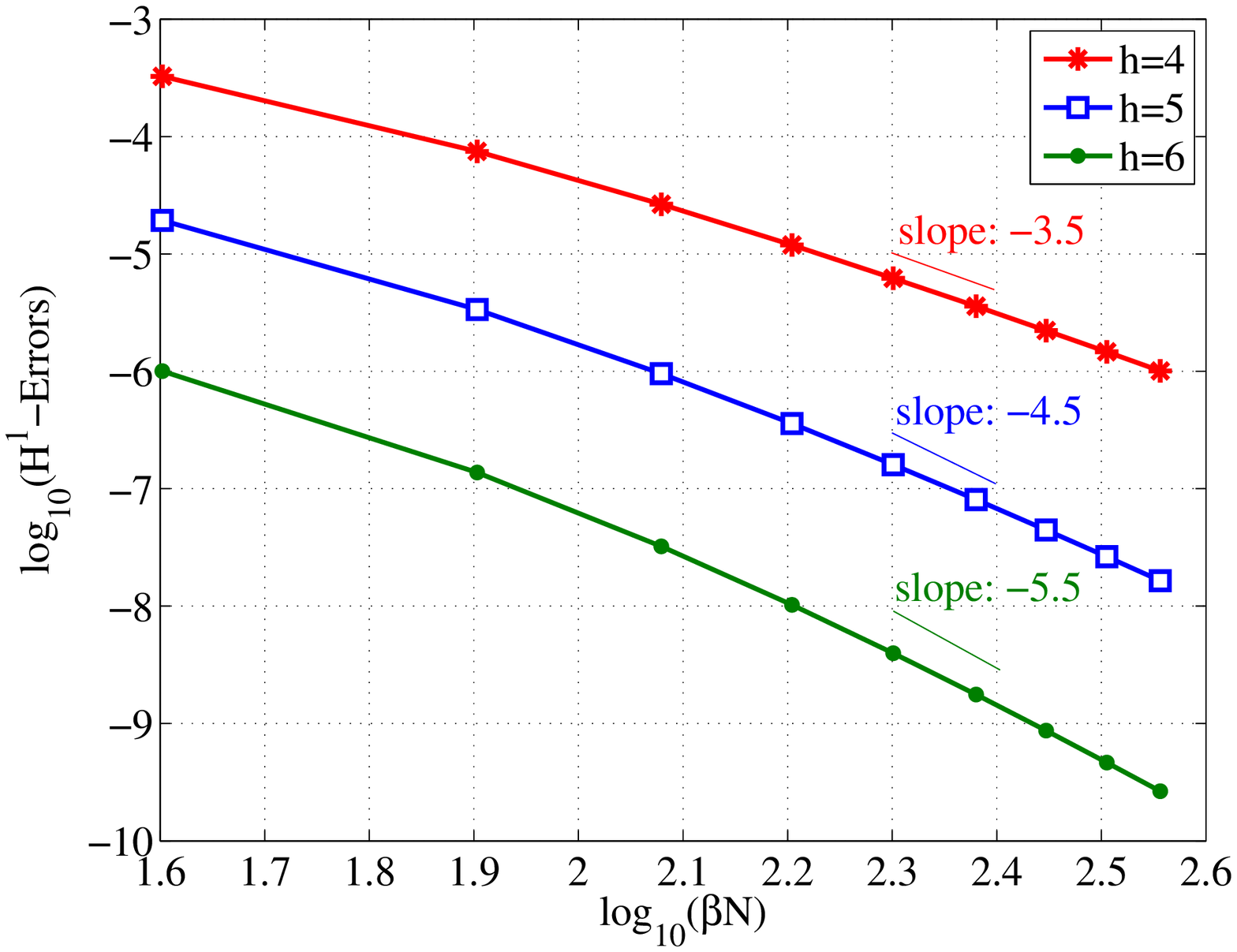}
\caption{\footnotesize $H^1$-errors with $\beta=4$.\label{fig10}}
\end{minipage}
\end{figure}

\begin{figure}[htbp]
\begin{minipage}[t]{0.45\linewidth}
\centering
\includegraphics [height=2.10in]{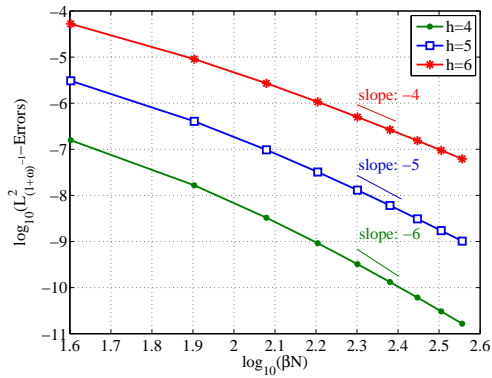}
\caption{\footnotesize $L^2_{(1+w)^{-1}}$-errors with
$\beta=4$.\label{fig9.1}}
\end{minipage}
\end{figure}
\subsection{On the condition numbers}
To demonstrate the essential  superiority  of our diagonalized Laguerre spectral method
to the  classic Laguerre  methods,
we finally  examine the issue on condition numbers for the resulting algebraic systems.
The diagonalized Laguerre spectral method use the Sobolev-orthogonal Laguerre  functions
$\{\mathcal{S}_{k}^{\beta}(x)/\sqrt{\varrho_k} \}_{k=1}^N$  and $\{\mathcal{R}_{k}^{\beta}(x)/\sqrt{\rho_k} \}_{k=0}^N$ as the basis functions
for \eqref{4.3s} and \eqref{5.3}, respectively.  All the
 condition numbers of the corresponding total stiff matrices  are equal to $1$. While in
the classical Laguerre spectral method,
the basis functions for \eqref{4.3s} and \eqref{5.3} are chosen as $\{xl_{k}^{0,\beta}(x)\}_{k=0}^{N-1}$
and $\{l_{k}^{0,\beta}(x)\}_{k=0}^{N}$, respectively. The  corresponding total stiff matrices have off-diagonal entries.

In Table \ref{Tab3.1} below, we list the condition numbers of the total stiff  matrices
of the  classical Laguerre spectral method for \refe{4.1s} with  $\gamma=1$ versus various $N$ and $\beta$.
The condition numbers of the  classical Laguerre spectral method for \refe{5.1} with  $\gamma=1$ and $\mu=1$
are tabulated in Table \ref{Tab3.2}.
We note  that the condition numbers of the resulting systems  increase asymptotically as  $\mathcal{O}(N^2).$

\begin{center}
\begin{table}[h]\begin{minipage}{0.5\linewidth}
\begin{tabular}{|c|c|c|c|}
\hline
N&$\beta=1$&$\beta=2$&$\beta=3$\\
\hline
10&1.5899E+03&5.1603E+02&2.8771E+02\\
30&2.3557E+04&6.8162E+03&3.5015E+03\\
50&7.5948E+04&2.1307E+04&1.0654E+04\\
70&1.6112E+05&4.4494E+04&2.1921E+04\\
90&2.8034E+05&7.6647E+04&3.7394E+04\\
110&4.3446E+05&1.1795E+05&5.7137E+04\\
130&6.2409E+05&1.6853E+05&8.1197E+04\\
150&8.4971E+05&2.2850E+05&1.0961E+05\\
\hline
\end{tabular}
\caption{Condition numbers of the classic Laguerre spectral method
for \eqref{4.1s} with $\gamma=1$.}\label{Tab3.1}\end{minipage}
\end{table}
\end{center}

\begin{center}
\begin{table}[h]\begin{minipage}{0.5\linewidth}
\begin{tabular}{|c|c|c|c|}
\hline
N&$\beta=1$&$\beta=2$&$\beta=3$\\
\hline
10&5.8863E+01&2.1075E+02&4.4972E+02\\
30&4.1557E+02&1.6049E+03&3.5610E+03\\
50&1.0965E+03&4.2960E+03&9.5913E+03\\
70&2.1016E+03&8.2840E+03&1.8540E+04\\
90&3.4309E+03&1.3569E+04&3.0407E+04\\
110&5.0845E+03&2.0151E+04&4.5191E+04\\
130&7.0623E+03&2.8029E+04&6.2894E+04\\
150&9.3643E+03&3.7205E+04&8.3515E+04\\
\hline
\end{tabular}
\caption{Condition numbers of the classic Laguerre spectral method
for \eqref{5.1} with $\gamma=1$ and $\mu=1$.}\label{Tab3.2}\end{minipage}
\end{table}
\end{center}

\end{document}